\documentclass[10pt]{amsart}

\usepackage{amssymb}
\usepackage{amsbsy}
\usepackage{amscd}
\usepackage[mathscr]{eucal}
\usepackage{verbatim}
\usepackage{version}
\usepackage{color}
\usepackage[matrix,arrow]{xy}
\usepackage{graphicx}

\def\cal{\mathcal}
\def\Bbb{\mathbb}

\newenvironment{NB}{
\color{red}{\bf NB}. \footnotesize 
}{}
\excludeversion{NB}
\newenvironment{NB2}{
\color{blue}{\bf NB}. \footnotesize
}{}

\newcommand{\simto}{\xrightarrow{\ \sim\ }}

\newcommand{\GL}  {\operatorname{GL}}

\newcommand{\Pic}{\operatorname{Pic}}

\newcommand{\ch}{\operatorname{ch}}

\newcommand{\Coh}{\operatorname{Coh}}

\newcommand{\Hom}{\operatorname{Hom}}

\newcommand{\rk}{\operatorname{rk}}

\newcommand{\NS}{\operatorname{NS}}

\newcommand{\td}{\operatorname{td}}

\newcommand{\Hilb}{\operatorname{Hilb}}

\newcommand{\Amp}{\operatorname{Amp}}

\newcommand{\alg}{\operatorname{alg}}

\newcommand{\Stab}{\operatorname{Stab}}

\newcommand{\Mov}{\operatorname{Mov}}

\def\Stabd{\mathop{\Stab^{\dagger}}\nolimits}

\font\b=cmr10 scaled \magstep5
\def\bigzerou{\smash{\lower1.7ex\hbox{\b 0}}}

\numberwithin{equation}{section}
\theoremstyle{plain}
 \newtheorem{thm}{Theorem}[section]
 \newtheorem{lem}[thm]{Lemma}
 \newtheorem{prop}[thm]{Proposition}
 \newtheorem{cor}[thm]{Corollary}
 
 \newtheorem{prob}[thm]{Problem}
\theoremstyle{definition}
 \newtheorem{defn}[thm]{Definition}
 
\theoremstyle{remark}
 \newtheorem{rem}[thm]{Remark}
 \newtheorem{ex}[thm]{Example}
 
\begin{document}

\title{Birational maps of moduli spaces of stable sheaves on a K3 surface
induced by Fourier-Mukai transforms.}
\author{K\={o}ta Yoshioka}
\address{Department of Mathematics, Faculty of Science,
Kobe University,
Kobe, 657, Japan
}
\email{yoshioka@math.kobe-u.ac.jp}

\thanks{
The author is supported by the Grant-in-aid for 
Scientific Research (No. 23K03053, 26K06742), JSPS}

\begin{abstract}
We shall study birational maps of moduli spaces of stable sheaves on a K3 surface with Picard number 1
induced by Fourier-Mukai transforms.
\end{abstract}

\maketitle

\renewcommand{\thefootnote}{\fnsymbol{footnote}}

\section{Introcduction}


Let $X$ be a K3 surface over ${\Bbb C}$ and $H$ an ample divisor on $X$.
A Mukai lattice of $X$ consists of 
$H^*(X,{\Bbb Z}):=\bigoplus_{i=0}^2 H^{2i}(X,{\Bbb Z})$ 
and an integral bilinear form
$\langle\;\;,\;\; \rangle$ on 
$H^*(X,{\Bbb Z})$:
$$
\langle x_0+x_1+x_2 \varrho_X,y_0+y_1+y_2 \varrho_X \rangle:=
(x_1 \cdot y_1)-x_0 y_2-x_2 y_0  \in {\Bbb Z},
$$
where $x_1,y_1 \in H^{2}(X,{\Bbb Z})$, 
$x_0,x_2,y_0,y_2 \in {\Bbb Z}$ and $\varrho_X \in H^4(X,{\Bbb Z})$ 
is the fundamental cohomology class of $X$.
We also introduce the algebraic Mukai lattice as
the pair of $H^*(X,{\Bbb Z})_{\alg}:=
{\Bbb Z} \oplus \NS(X) \oplus {\Bbb Z}$
and $\langle\;\;,\;\; \rangle$ on 
$H^*(X,{\Bbb Z})_{\alg}$.
It is the $(1,1)$-part $H^{1,1}(H^*(X,{\Bbb C})) \cap H^*(X,{\Bbb Z})$
of the weight 2 Hodge structure on the Mukai lattice. 
For $x=x_0+x_1+x_2 \varrho_X$ with $x_0,x_2 \in {\Bbb Z}$
and $x_1 \in H^2(X,{\Bbb Z})$,
we also write $x=(x_0,x_1,x_2)$. 

Let ${\bf D}(X)$ be the bounded derived category of the category $\Coh(X)$ 
of coherent sheaves on $X$.
For an object $E \in {\bf D}(X)$, 
let $v(E):=\ch(E)\sqrt{\td_X} \in H^*(X,{\Bbb Z})_{\alg}$ be the 
Mukai vector of $E$.

\begin{defn}
A Mukai vector $v=(r,\xi,a) \in H^*(X,{\Bbb Z})_{\alg}$ is positive if
\begin{enumerate}
\item $r>0$ or 
\item $r=0$ and $\xi \ne 0$ is effective or
\item $r=\xi=0$ and $a>0$.
\end{enumerate}
We denote a positive Mukai vector $v$ by $v>0$.
\end{defn}

Let $H$ be an ample divisor on $X$.
For a positive Mukai vector $v \in H^*(X,{\Bbb Z})_{\alg}$ with $\langle v^2 \rangle \geq -2$,
$M_H(v)$ denotes the moduli space of stable sheaves $E$ with $v(E)=v$.

A Fourier-Mukai functor $\Phi:{\bf D}(X) \to {\bf D}(X')$ is an integral functor which gives an equivalence
of the derived categories.   
A contravariant Fourier-Mukai functor ${\bf D}(X) \to {\bf D}(X')$
is the composition ${\cal D}_{X'} \circ \Phi$
of a Fourier-Mukai functor $\Phi:{\bf D}(X) \to {\bf D}(X')$
and the derived dual functor ${\cal D}_{X'}(\quad):={\bf R}{\cal H}om_{{\cal O}_{X'}}(\quad,{\cal O}_{X'})$.
In \cite{BM}, Bayer and Macri studied moduli space $M_\sigma(v)$ of Bridgeland stable objects on $X$,
where $\sigma$ is a stability condition.
In particular they studied wall crossing behavior of the moduli spaces
under the change of stability conditions $\sigma$.
As an application, they obtained the following characterization of the birational
equivalence class of $M_H(v)$.

 \begin{thm}[cf. {\cite[Cor. 1.3]{BM}}]\label{thm:BM}
Let $(X,H)$ be a pair of K3 surface $X$ and an ample divisor $H$ on $X$
and  $(X',H')$ be a pair of K3 surface $X'$ and an ample divisor $H'$ on $X'$. 
Then the following holds.
\begin{enumerate}
\item[(1)]
$M_H(v)$ is birationally equivalent to $M_{H'}(v')$ if and only if 
there is a (contravariant) equivalence $\Phi:{\bf D}(X) \to {\bf D}(X')$ 
such that $\overline{\Phi}(v)=v'$, where
$\overline{\Phi}:H^*(X,{\Bbb Z}) \to H^*(X',{\Bbb Z})$ is the Hodge
isometry induced by $\Phi$.
\item[(2)]
Under the condition of (1),
we can choose $\Phi$ to satisfy that
$\Phi(E) \in M_{H'}(v')$ for a general $E \in M_H(v)$.
\end{enumerate} 
\end{thm}

\begin{NB}
For a given (contravariant) Fourier-Mukai transform $\Phi$,
Bayer and Macri \cite{BM} prove that there is a (contravariant) Fourier-Mukai transform
$\Psi:{\bf D}(X') \to {\bf D}(X')$ such that 
$\Psi \circ \Phi$ induces a birational map of moduli spaces of stable sheaves.
\end{NB}

In this note, we shall prove that birational maps of moduli spaces of stable sheaves are induced by
(contravariant) Fourier-Mukai transforms, if the Picard number of $X$ is 1.

\begin{thm}[Theorem \ref{thm:bir-auto}]\label{thm:intro:bir-auto}
Let $(X,H)$ and $(X',H')$ be polarized K3 surfaces.
Let $\varphi:M_H(v) \cdots \to M_{H'}(v')$ be a birational
map, where $v \in H^*(X,{\Bbb Z})$ and $v' \in H^*(X',{\Bbb Z})$ are positive Mukai vectors.
Assume that $\Pic(X)={\Bbb Z}H$.
Then there is a (contravariant) Fourier-Mukai transform
$\Phi:{\bf D}(X) \to {\bf D}(X')$ 
which induces $\varphi_*$.
\end{thm}

For the birational map $\varphi$, we have a Hodge isometry 
$v^\perp \to {v'}^\perp$, which can be extended to a Hodge isometry
$f:H^*(X,{\Bbb Z}) \to H^*(X',{\Bbb Z})$ by Markman's  computation of the monodromy groups
of $M_H(v)$ \cite{Markman}.
Then we have a (contravariant) Fourier-Mukai transform
$\Phi_1:{\bf D}(X) \to {\bf D}(X')$ which is a lifting of $f$.
Bayer and Macri \cite{BM} prove that there is a (contravariant) Fourier-Mukai transform
$\Psi:{\bf D}(X') \to {\bf D}(X')$ (depending on the structure of walls)
such that 
$\Phi:=\Psi \circ \Phi_1$ induces a birational map of moduli spaces of stable sheaves.
It is not known that we can choose $\Psi$ acting trivially on the Mukai lattice.
Theorem \ref{thm:intro:bir-auto} implies that we can choose $\Psi$ which acts trivially on the Mukai lattice,
provided that $\Pic(X)={\Bbb Z}H$. 

For the rank 1 case, Gieseker stability is independent of the choice of polarization.
Hence we also get the following result. 
\begin{thm}[Theorem \ref{thm:bir-auto2}]
Assume that there is no $(-2)$-curves.
Then birational automorphisms of $\Hilb_X^{\ell+1}$ are induced by (contravariant) Fourier-Mukai transforms.  
\end{thm}

Let us briefly explain the outline of the proof.
Let $\varphi:M_H(v) \cdots \to M_{H'}(v')$ be a birational map of moduli spaces
and $f:H^*(X,{\Bbb Z}) \to H^*(X',{\Bbb Z})$ the associated Hodge isometry.
Assume that $f$ preserves the orientation.
Then we have a Fourier-Mukai transform $\Phi_{X \to X'}^{{\cal E}^{\vee}[k]}:{\bf D}(X) \to {\bf D}(X')$
which induces $f$.
If there is no $(-2)$-curve on $X'$,
then we can choose $X'=M_H(v_0)$ and 
${\cal E}$ is a universal family of stable sheaves on $X$ (Lemma \ref{lem:lifting}), 
where $v_0 \in H^*(X,{\Bbb Z})$ is a primitive and isotropic Mukai vector.
For simplicity,
we assume that ${\cal E}$ is a family of $\mu$-stable locally free sheaves.
Then the problem is reduced to show that $\Phi_{X \to X'}^{{\cal E}^{\vee}[k]}$
induces a birational map $M_H(v) \cdots \to M_{H'}(v')$ which is regular in codimension 1.
For this problem, the theory of Bridgeland stability conditions is very useful.
Let us consider geometric stability conditions $\sigma_{(\beta,\omega)}$ such that
${\Bbb C}_x$ ($x \in X$) are stable objects with phase 1,
where the parameters $(\beta,\omega)$ belong to an open subset ${\cal H}$ of
$\NS(X)_{\Bbb R} \times \Amp(X)_{\Bbb R}$, and
we have a map $\xi_v:{\cal H} \to (v^\perp)_{\Bbb R} \cong H^2(M_H(v),{\Bbb R})$
(cf. subsection \ref{subsect:line-bundle}).
\begin{NB}
Let us consider geometric stability conditions $\sigma_{(\beta,\omega)}$,
where $\beta \in \NS(X)_{\Bbb R}$ and $\omega \in \Amp(X)_{\Bbb R}$.
They are parameterized by an open subset of $\NS(X)_{\Bbb R} \times \Amp(X)_{\Bbb R}$, and
we have a map to $v^\perp \cong H^2(M_H(v),{\Bbb R})$.
\end{NB}
Then we can relate the movable cone of $M_H(v)$ with chambers for stability conditions
in ${\cal H}$ via $\xi_v$. 
We shall study wall crossing behaviors of the stability conditions along the curve
$C:=\{(\gamma,tH) \mid t>0 \} \subset {\cal H}$, where $v_0=r_0 e^\gamma$.
Let $E$ be a $\sigma_{(\gamma,tH)}$-stable object with $v(E)=v$.
If $t \gg 0$ (i.e., large volume limit), then it is well-known that
(1) $E$ is a stable sheaf or (2) $E$ is not a stable sheaf but
$E^{\vee}$ is a stable sheaf. 
We set $F:=\Phi_{X \to X'}^{{\cal E}^{\vee}[k]}(E)$. 
If $t \ll 1$, then (1) $F$ is a stable sheaf or (2) 
$F$ is not a stable sheaf and $F$ is the dual of a stable sheaf.
By Bayer and Macri's description of the movable cone of the moduli spaces
and some constraints obtained in subsection \ref{subsect:r=1,2},
we see that (2) does not occur for a general $E$.
Since $f$ preserves the movable cones,
Bayer and Macri's description of the movable cone also implies
that all walls intersecting the curve $C$
are totally semistable walls or walls corresponding to small contractions. 
Hence there is an autoequivalence $\Psi$ of ${\bf D}(X')$ which acts trivially on the Mukai lattice and
$\Phi=\Psi \circ \Phi_{X \to X'}^{{\cal E}^{\vee}[k]}$ induces a desired birational map. 
In general we can not assume that ${\cal E}$ is a family of $\mu$-stable sheaves.
If ${\cal E}$ is a family of stable torsion free sheaves, then
by studying wall crossing behavior along a perturbed curve $\{(\gamma+s_0 H,tH) \mid t>0 \}$,
we also get our results.
If ${\cal E}$ is a family of stable 1-dimensional sheaves, then
we can reduce to the former case by using a suitable Fourier-Mukai transform, and get our results.

A similar result is obtained by Liu \cite{Liu}.
In the paper, he concentrated in the case of rank 1
and treated the birational automorphisms.
Then he can use the Hodge isometry of Mukai lattice in \cite{BC}
to compute $\Phi$ explicitly.
In the general case, due to the existence of totally semistable walls defined by spherical objects,
we cannot specify $\Phi$ (cf. \cite{FM118}).  
A similar claim also holds for the moduli of stable sheaves on an abelian surface \cite{Bir-auto}.
For an abelian surface $X$, 
stability conditions are parameterized by the whole of $\NS(X)_{\Bbb R} \times \Amp(X)_{\Bbb R}$.
Due to this simple description of the parameter space of stability conditions,
we can show the corresponding results without any restriction.

In section \ref{sect:stable sheaf},
we collect some properties of moduli spaces on $M_H(v)$.
In particular
we explain the second cohomology group of $M_H(v)$ and some ${\Bbb Q}$-divisors on $M_H(v)$.  
In \ref{subsect:stableFM},
we remark that the orientation preserving isometry of Mukai lattices 
are lifted to an Fourier-Mukai transform by a family of stable sheaves (Lemma \ref{lem:lifting}).

In section \ref{sect:Bridgeland}, we recall Bridgeland stability conditions.
In particular we recall the conditions for the upper half plane ${\Bbb H}$ 
such that $(s,t) \in {\Bbb H}$ defines a stability condition 
$\sigma_{(\gamma+sH,tH)}$.

In section \ref{sect:FM},
we discuss on the action of Fourier-Mukai transforms on the our space of stability conditions. 
In \ref{subsect:Bridgeland-FM}, we recall the behavior of stability conditions
under Fourier-Mukai transforms.
In particular we shall discuss its relation with divisors
in subsection \ref{subsect:f}. 
Then we give some constraints on $v$ 
such that the equivalence induces a birational map of moduli spaces
in subsection \ref{subsect:r=1,2}.

In section \ref{sect:birat},
we prove Theorem \ref{thm:bir-auto}.
In order to reduce to the Fourier-Mukai transform by a family of stable torsion free sheaves,
we first study a Fourier-Mukai transform by a family of stable 1-dimensional sheaves.
Then we show that the Fourier-Mukai transform preserving the movable cone of moduli spaces
induces birational map, these are done in in subsection \ref{subsect:bir-FM}.
In subsection \ref{subsect:main}, we prove Theorem \ref{thm:bir-auto}.

\section{Moduli of stable sheaves.}\label{sect:stable sheaf}
{\bf Notation.}
\begin{enumerate}
\item
For a lattice $L$, $A_L:=L^{\vee}/L$ denotes the discriminant group of $L$.
\item
For a smooth projective variety $X$,
$\Amp(X) (\subset \NS(X)_{\Bbb Q})$ is the set of ample ${\Bbb Q}$-divisors,
$\Amp(X)_{\Bbb R} (\subset \NS(X)_{\Bbb R})$ is the ample cone
and $\Mov(X) (\subset \NS(X)_{\Bbb R})$ is the movable cone of $X$.
\item
For a point $x$ of a scheme $X$, 
${\Bbb C}_x \in \Coh(X)$ denotes the structure sheaf of 
$x$.
\item
For a coherent sheaf $E$ on a smooth variety $X$, $\rk E$ denotes the rank of $E$.
For a bounded complex $E$ of coherent sheaves on $X$, 
$\rk E:=\sum_i (-1)^i \rk H^i(E)$
denotes the rank of $E$, where $H^i(E) \in \Coh(X)$ is the $i$-th cohomology sheaf of $E$.  
For $v \in H^*(X,{\Bbb Q})$,
$\rk v \in {\Bbb Q}$ denotes the $H^0(X,{\Bbb Q})$-component of
$v$. 
\begin{NB}
\item
For a K3 surface $X$ with $\NS(X) \cong {\Bbb Z}$,
$H_X$ denotes the ample generator of $\NS(X)$.
We also use the symbol $H$ as the ample generator of $\NS(X)$, if
we do not need to notify the underlying space $X$. 
\end{NB}
\item
\begin{NB}
For a K3 surface $X$, we define $1_X \in H^*(X,{\Bbb Z})$ and 
$\varrho_X \in H^*(X,{\Bbb Z})$ by
$$
1_X:=(1,0,0),\;\varrho_X:=(0,0,1).
$$ 
\end{NB}
Let $X$ be a K3 surface.
For $(x_0,x_1,x_2) \in H^*(X,{\Bbb Z})_{\alg}$, we set
$$
(x_0,x_1,x_2) e^\xi:=(x_0,x_1+x_0 \xi,x_2+(x_1 \cdot \xi)+x_0 (\xi^2)/2), \xi \in \NS(X)_{\Bbb Q}.
$$
We also set $e^\xi:=1_X e^\xi=(1,\xi,(\xi^2)/2)$.
\begin{NB}
\item
For a ${\Bbb Z}$-module ${\Bbb L}$ (for example $H^*(X,{\Bbb Z})_{\alg}$)
and $v \in {\Bbb L}$, $\overline{v} \in {\Bbb L}/\{ \pm 1 \}$ denotes the equivalence class 
$\{\pm v \}$ represented by $v$.
\end{NB}
\item
Let $X$ and $X'$ be K3 surfaces.
For an object ${\bf E} \in {\bf D}(X \times X')$,
we have an integral functor 
\begin{equation}
\begin{matrix}
\Phi_{X \to X'}^{{\bf E}}:& {\bf D}(X) & \to & {\bf D}(X')\\
& E & \mapsto & {\bf R}p_{X'*}(p_X^*(E) \otimes {\bf E})
\end{matrix}
\end{equation}
where $p_X:X \times X' \to X$ and $p_{X'}:X \times X' \to X'$ are projections.
If $\Phi_{X \to X'}^{{\bf E}}$ is an equivalence, then
it induces a Hodge isometry 
$\overline{\Phi}_{X \to X'}^{{\bf E}}$ 
of the Mukai lattices with a commutative diagram:
\begin{equation}
\begin{CD}
{\bf D}(X) @>{\Phi_{X \to X'}^{{\bf E}}}>> {\bf D}(X')\\
@V{v}VV @VV{v}V\\
H^*(X,{\Bbb Z}) @>>{\overline{\Phi}_{X \to X'}^{{\bf E}}}> H^*(X',{\Bbb Z})
\end{CD}
\end{equation}
where the vertical arrows $v$ are defined by taking the Mukai vectors.
For an equivalence $\Phi:{\bf D}(X) \to {\bf D}(X')$,
$\overline{\Phi}$ denotes the Hodge isometry $\overline{\Phi}_{X \to X'}^{\bf E}$,
where $\Phi=\Phi_{X \to X'}^{\bf E}$.

\item
Let $X$ be a K3 surface.
For an object $E \in {\bf D}(X)$, we set
$$
{\cal D}_X(E):={\bf R}{\cal H}om_{{\cal O}_X}(E,{\cal O}_X).
$$
Then
${\cal D}_X:{\bf D}(X) \to {\bf D}(X)$ is a contravariant functor and
${\cal D}_X$ induces a Hodge isometry
\begin{equation}
\begin{matrix}
\overline{\cal D}_X: & H^*(X,{\Bbb Z}) & \to & H^*(X,{\Bbb Z})\\
& (x_0,x_1,x_2) & \mapsto & (x_0,-x_1,x_2).
\end{matrix}
\end{equation}
For ${\cal D}_X$ and $\overline{\cal D}_X$,
we have the relation  
$\overline{\cal D}_X \circ v=v \circ {\cal D}_X$.
\end{enumerate}

\subsection{Some properties of moduli spaces.}\label{subsect:properties}
Let $X$ be a K3 surface and $v$ be a positive Mukai vector.
If $v$ is primitive and $H$ is a general ample divisor, then
$M_H(v)$ is an irreducible symplectic manifold of dimension $\langle v^2 \rangle+2$
which is deformation equivalent to the Hilbert scheme of points \cite{Y:7}.
\begin{NB}
For the birational equivalence classes of $M_H(v)$, combining Markman's results on the monodromy groups
of $M_H(v)$ \cite{Markman}, Bayer and Macri \cite{BM}
give the following general result. 

 \begin{thm}[cf. {\cite[Cor. 1.3]{BM}}]\label{thm:BM}
Let $(X,H)$ be a pair of K3 surface $X$ and an ample divisor $H$ on $X$
and  $(X',H')$ be a pair of K3 surface $X'$ and an ample divisor $H'$ on $X'$. 
Then the following holds.
\begin{enumerate}
\item[(1)]
$M_H(v)$ is birationally equivalent to $M_{H'}(v')$ if and only if 
there is a (contravariant) equivalence $\Phi:{\bf D}(X) \to {\bf D}(X')$ 
such that $\overline{\Phi}(v)=v'$, where
$\overline{\Phi}:H^*(X,{\Bbb Z}) \to H^*(X',{\Bbb Z})$ is the Hodge
isometry induced by $\Phi$.
\item[(2)]
Under the condition of (1),
we can choose $\Phi$ to satisfy that
$\Phi(E) \in M_{H'}(v')$ for a general $E \in M_H(v)$.
\end{enumerate} 
\end{thm}
\end{NB}

\begin{defn}\label{defn:theta}
Assume that $v$ is primitive and $\langle v^2 \rangle \geq 0$.
\begin{enumerate}
\item[(1)]
\begin{equation}\label{eq:theta}
\theta_{v}:v^{\perp} \to H^2(M_{H}(v),{\Bbb Z})
\end{equation}
denotes the Mukai's homomorphism.
\item[(2)]
If there is a universal family
${\bf E}$ on $M_{H}(v)$, e.g.,
there is a Mukai vector $w$ with
$\langle v,w \rangle=1$,
then 
$$
\theta_{v}(x)=
c_1(p_{M_{H}(v)*}
(\ch ({\bf E})p_X^*(1+\varrho_X) p_X^*(x^{\vee})))=
-c_1(p_{M_{H}(v)*}(\ch ({\bf E}^{\vee}) p_X^*(1+\varrho_X) p_X^*(x))).
$$ 
where $p_X, p_{M_H(v)}$ are projections
from $X \times M_H(v)$ to $X$ and
$M_H(v)$ respectively. 
\end{enumerate}
\end{defn}

\begin{thm}[{\cite{Mukai:Tata},\cite{Y:7}}]
\begin{enumerate}
\item[(1)]
If $\langle v^2 \rangle >0$, then
$\theta_v$ is a Hodge isometry, where the lattice structure on
$H^2(M_H(v),{\Bbb Z})$ is defined by Beauville-Bogomolov form.
\item[(2)]
If $\langle v^2 \rangle=0$, then 
$\theta_v$ induces a Hodge isometry
$$
v^\perp/{\Bbb Z}v \to H^2(M_H(v),{\Bbb Z}).
$$ 
\end{enumerate}
\end{thm}

Bayer and Macri studied the birational geometry of $M_H(v)$ by using Bridgeland stability conditions.
In particular they get the following description of the movable cone of $M_H(v)$. 
\begin{thm}[{\cite[Thm. 12.3]{BM}}]\label{thm:BM2}
Under the identification \eqref{eq:theta},
the boundaries $\partial \Mov(M_H(v))$ of $\Mov(M_H(v))$ are defined by $u^\perp$, where
$u \in H^*(X,{\Bbb Z})_{\alg}$ satisfies one of the following conditions.
\begin{enumerate}
\item
$\langle u^2 \rangle=0$ and $\langle v,u \rangle=1,2$.   
\item
$\langle u^2 \rangle=-2$ and $\langle u,v \rangle=0$.
\end{enumerate}
\end{thm}

\begin{rem}\label{rem:BM2}
For the Mukai vector $u$ in Theorem \ref{thm:BM2},
we have a wall $W_u$ in the space of stability conditions $\Stab(X)$.
If $\langle u^2 \rangle=0$ and $\langle v,u \rangle=1$,
$W_u$ is a totally semistable wall, that is,
all objects are properly $\sigma$-semistable with respect to $\sigma \in W_u$.
Another type of totally semistable wall is defined by an spherical object
$F$ with $\langle v,v(F) \rangle<0$.
For this type of totally semistable wall,
there is an autoequivalence $\Psi$ which acts trivially on the Mukai lattice and
induces an isomorphism $M_{\sigma_-}(v) \to M_{\sigma_+}(v)$, where 
$M_{\sigma_\pm} (v)$ are moduli spaces of $\sigma_\pm$-stable objects such that
$\sigma_\pm$ belong to opposite chambers of the wall.  
\end{rem}

\subsection{Divisors on $M_H(v)$.}\label{subsect:divisor}

Let $X$ be a K3 surface and
$$
v=r+\xi+a \varrho_X,\; r \in {\Bbb Z}_{>0}, \xi \in \NS(X), a \in {\Bbb Z}
$$
be a primitive Mukai vector with $\langle v^2 \rangle>0$.
For $\gamma \in \NS(X)_{\Bbb Q}$,
we set
$$
\delta:=\frac{\xi}{r}-\gamma \in \NS(X)_{\Bbb Q},\;
a_\gamma:=a-(\xi \cdot \gamma)+r\frac{(\gamma^2)}{2} \in {\Bbb Q}.
$$
Then
\begin{equation}
v=e^\gamma (r+r\delta+a_\gamma \varrho_X).
\end{equation}
\begin{defn}
For the Mukai vector $v$ and $\gamma \in \NS(X)_{\Bbb Q}$,
we set
\begin{equation}
h_v:=e^{\xi/r}H=H+\frac{(\xi\cdot H)}{r} \varrho_X,\;
D_{v,\gamma}:=-re^\gamma+a_\gamma \varrho_X.
\end{equation}
\end{defn}
We set
\begin{equation}
L:={\Bbb Q}h_v+{\Bbb Q}D_{v,\gamma}.
\end{equation}
We note the following results.
\begin{lem}\label{lem:ample}
\begin{enumerate}
\item[(1)]
$\theta_v(h_v+\epsilon D_{v,\gamma}) \in \NS(M_H^\gamma(v))_{\Bbb Q}$ $(0<\epsilon \ll 1)$ is ample,
\item[(2)]
$\theta_v(h_{v})$ is nef and big. In particular it is movable.
\end{enumerate}
\end{lem}

\begin{lem}\label{lem:mov}
Assume that $r=1,2$.
\begin{enumerate}
\item[(1)]
$\theta_v(h_v)$ defines a divisorial contraction of $M_H(v)$.
\item[(2)]
$\Mov(M_H(v)) \cap L$ is spanned by
$\theta_v(h_v)$ and $\theta_v(xh_v+D_{v,\gamma})$, where $x \in {\Bbb Q}$.
\end{enumerate}
\end{lem}

\begin{proof}
(1)
$\theta_v(\varrho_X+\frac{r}{\langle v^2 \rangle}v)^\perp$ is the wall of $\Mov(M_H(v))$
defining the Hilbert-Chow contraction or the Donaldson-Uhlenbeck contraction. 
Since $h_v \in \varrho_X^\perp$, $\theta_v(h_v)$ defines a divisorial contraction.

(2) By Lemma \ref{lem:ample},
$\Mov(M_H(v)) \cap L$ is spanned by
$\theta_v(h_v)$ and $\theta_v(xh_v+D_{v,\gamma})$ ($x \in {\Bbb Q}$).
\end{proof}

\subsection{Fourier-Mukai transforms by a family of stable sheaves.}\label{subsect:stableFM}

\begin{lem}\label{lem:lifting}
Let $X$ and $Y$ be K3 surfaces.
Assume that there is no $(-2)$-curves on $Y$.
\begin{enumerate}
\item[(1)]
Let $f:H^*(X,{\Bbb Z}) \to H^*(Y,{\Bbb Z})$ be an orientation preserving Hodge isometry.
Then there is an equivalence  
$\Phi_{X \to Y}^{{\cal E}^{\vee}[k]}:{\bf D}(X) \to {\bf D}(Y)$
such that ${\cal E}$ is a family of stable sheaves on $X$ 
and $\Phi_{X \to Y}^{{\cal E}^{\vee}[k]}$ induces $f$.
\item[(2)]
Let $f:H^*(X,{\Bbb Z}) \to H^*(Y,{\Bbb Z})$ be an orientation reversing Hodge isometry.
Then there is an equivalence  
$\Phi_{X \to Y}^{{\cal E}[k]}:{\bf D}(X) \to {\bf D}(Y)$
such that ${\cal E}$ is a family of stable sheaves on $X$ 
and ${\cal D}_{Y} \circ \Phi_{X \to Y}^{{\cal E}[k]}=\Phi_{X \to Y}^{{\cal E}^{\vee}[2-k]} \circ {\cal D}_X$ induces $f$.
\item[(3)]
Let $f:H^*(X,{\Bbb Z}) \to H^*(Y,{\Bbb Z})$ be an orientation reversing Hodge isometry.
Then there is an equivalence  
$\Phi_{X \to Y}^{{\cal E}^{\vee}[k]}:{\bf D}(X) \to {\bf D}(Y)$
such that ${\cal E}$ is a family of stable sheaves on $X$ 
and ${\cal D}_Y \circ \Phi_{X \to Y}^{{\cal E}^{\vee}[k]}$ induces $f$.
\end{enumerate}
\end{lem}

\begin{proof}
(1)
Let $f:H^*(X,{\Bbb Z}) \to H^*(Y,{\Bbb Z})$ be an orientation preserving Hodge isometry.
We take $u \in H^*(X,{\Bbb Z})_{\alg}$ with $f(u)=\varrho_{Y}$.
Replacing $f$ by $-f$, we assume that $u>0$.
We set $X':=M_H(u)$, where $H$ is a general ample divisor.
Since $\langle f^{-1}(v({\cal O}_X),u \rangle=\langle v({\cal O}_Y),\varrho_X \rangle=-1$,
there is a universal family ${\cal E}$ on $X \times X'$.
Let $\Phi_{X \to X'}^{{\cal E}^{\vee}}:{\bf D}(X) \to {\bf D}(X')$ be an equivalence
such that $\overline{\Phi}_{X \to X'}^{{\cal E}^{\vee}}(u)=\varrho_{X'}$.
Since $\Phi_{X \to X'}^{{\cal E}^{\vee}}$ is orientation preserving (\cite{HMS}),
$\varphi:=\overline{\Phi}_{X \to X'}^{{\cal E}^{\vee}} \circ f^{-1}$ is
also orientation preserving. Thus
$\varphi:H^*(Y,{\Bbb Z}) \to H^*(X',{\Bbb Z})$ is an orientation preserving Hodge isometry
such that 
$\varphi(\varrho_Y)=\varphi_{X'}$ and
$\varphi(v({\cal O}_Y))=v({\cal O}_{X'}(D))$ $(D \in \NS(X'))$.
Replacing ${\cal E}$ by ${\cal E}(D)$,
we get $\varphi(v({\cal O}_Y))=v({\cal O}_{X'})$.
Then we have a Hodge isometry
$H^2(Y,{\Bbb Z}) \to H^2(X',{\Bbb Z})$ preserving positive cone, 
which shows that $Y \cong X'$.
Since there is no $(-2)$-curves on $Y$, $\varphi$ is induced by an isomorphism
$\psi:Y \to X'$.
Thus ${\cal E}_1:=(1_X \times \psi)^*({\cal E})$
defines a desired Fourier-Mukai transform
$\Phi_{X \to Y}^{{\cal E}_1^{\vee}}$.
\begin{NB}
$\Phi_{X \to Y}^{{\cal E}_1^{\vee}}=\psi^* \circ \Phi_{X \to X'}^{{\cal E}^{\vee}}$.
\end{NB}

(2)
Since $f \circ \overline{\cal D}_{X}$ is orientation preserving, the claim follows from (1).

(3)
Since $\overline{\cal D}_Y \circ f$ is orientation preserving, 
the claim also follows from (1).
\end{proof}

\begin{NB}
Assume that ${\cal E}$ is a family of Gieseker stable sheaves with respect to $H$.
Then $c_1(\Phi_{X \to X'}^{{\cal E}^{\vee}[1]}(H u))$ is ample.  

If $f(L u)$ is not ample for all $L \in \Amp(X)$, then 
$f$ is not lifted to an equivalence defined by a family of Gieseker semistable sheaves.

For $u=(r_0,\xi_0,a_0)=r_0 e^\gamma$, we have an isometry
\begin{equation}
\begin{matrix}
\psi_\gamma:& \NS(X)_{\Bbb Q} & \to & u^\perp \cap \varrho_X^\perp\\
& D & \mapsto & De^\gamma 
\end{matrix}
\end{equation}

$\psi_\gamma (\Amp(X)_{\Bbb Q}) \cap \Amp(Y)_{\Bbb Q} \ne \emptyset$,
then there is a family of Gieseker stable sheaves which induces $f$.

Let $g$ be an isometry such that
$g(\Amp(Y)) \cap \Amp(Y)=\emptyset$.
If $f$ is induced by a family of Gieseker stable sheaves, then
$f \circ g$ is not induced by a family of Gieseker stable sheaves.
 
For a $(-2)$-curve $C$,
we take an integer $n$ such that
$0<(\gamma' \cdot C)-n<1$.
Then $0<\langle e^{\gamma'}, v({\cal O}_C(n-1)) \rangle<1$.

$$
0 \to {\cal O}_C(n) \to k_y \to {\cal O}_C(n-1)[1] \to 0.
$$
Applying $\Phi_{Y \to X}^{{\cal E}[1]}$, we have an exact sequence
defining a Bridgeland wall for ${\cal E}_{|X \times \{y\}}$.
Thus $C$ corresponds to a Bridgeland wall for $u$.
\end{NB}

\section{Bridgeland stability conditions on a K3 surface.}\label{sect:Bridgeland}

\subsection{Stability conditions}\label{subsect:stab-cond}
Let us briefly recall some notation on stability conditions on a K3 surface.
For more details, see \cite{Br:stability}, \cite{Br:3} and \cite{Br:4}.
A stability condition $\sigma=(Z_\sigma,{\cal P}_\sigma)$ 
on ${\bf D}(X)$
consists of a group homomorphism $Z_\sigma: {\bf D}(X) \to {\Bbb C}$
and a slicing ${\cal P}_\sigma$ of ${\bf D}(X)$ such that if
$0 \ne E \in {\cal P}_\sigma(\phi)$ then 
$Z_\sigma(E) =m(E)\exp(\pi \sqrt{-1} \phi)$ 
for some $m(E) \in {\Bbb R}_{>0}$.
The set of stability conditions has a structure of complex 
manifold.
We denote this space by $\Stab(X)$.

\begin{defn}\label{defn:B-stable}
\begin{enumerate}
\item[(1)]
A $\sigma$-semi-stable object $E$ of phase $\phi$ is
an object of  ${\cal P}_\sigma(\phi)$.
If $E$ is a simple object of ${\cal P}_\sigma(\phi)$, then
$E$ is $\sigma$-stable. 
\item[(2)]
For $v \in H^*(X,{\Bbb Z})_{\alg}$,
we denote the moduli stack of $\sigma$-semi-stable objects $E$ with 
$v(E)=v$ by ${\cal M}_\sigma(v)$,
where we usually choose 
$\phi:=\mathrm{Im}(\log Z_\sigma(v))/\pi \in (-1,1]$.
\end{enumerate}
\end{defn}

By \cite[Prop. 5.3]{Br:stability},
giving a stability condition $\sigma$ is
the same as giving a bounded $t$-structure on ${\bf D}(X)$ 
and a stability function $Z_\sigma$ on its heart ${\cal A}_\sigma$
with the Harder-Narasimhan property.
For $\sigma=(Z_\sigma,{\cal P}_\sigma)$,
we have the relation ${\cal A}_\sigma ={\cal P}_\sigma((0,1])$,
where ${\cal P}_\sigma((0,1])$ is the subcategory of ${\bf D}(X)$
generated by semi-stable objects $E \in {\cal P}_\sigma(\phi)$ with 
$\phi \in (0,1]$.  
Since the pair $(Z_\sigma,{\cal A}_\sigma)$
defines a stability condition, 
we also use the symbol $\sigma=(Z_\sigma,{\cal A}_\sigma)$
to denote a stability condition.
For a $\sigma$-semi-stable object $E \in 
{\cal P}_\sigma(\phi)$, we set $\phi_\sigma(E)=\phi$.
For $E \in {\cal P}_\sigma((0,1])$,
we also set $\phi_\sigma(E) \in (0,1]$ 
by $Z_\sigma(E)=m(E)\exp(\pi \sqrt{-1} \phi_\sigma(E))$. 

For an equivalence $\Phi:{\bf D}(X) \to {\bf D}(X')$,
we have an isomorphism $\Phi:\Stab(X) \to \Stab(X')$
such that 
$\Phi(\sigma)$ $(\sigma \in \Stab(X))$ 
is a stability condition given by
\begin{equation}\label{eq:FM-action}
\begin{split}
Z_{\Phi(\sigma)}=& Z_\sigma \circ \Phi^{-1}:{\bf D}(X') \to {\Bbb C},\\
{\cal P}_{\Phi(\sigma)}(\phi)=& \Phi({\cal P}_\sigma(\phi)).
\end{split}
\end{equation}
We also have an action of the universal covering 
$\widetilde{\GL}_2^+({\Bbb R})$ of $\GL_2^+({\Bbb R})$ on 
$\Stab(X)$.
Since ${\Bbb C}^{\times} \subset  \GL_2^+({\Bbb R})$,
we have an injective homomorphism
${\Bbb C} \to  \widetilde{\GL}_2^+({\Bbb R})$.
Thus 
we have an action of $\lambda \in {\Bbb C}$ on
$\Stab(X)$. 
For a stability condition $\sigma \in \Stab(X)$, 
$\lambda(\sigma)$ is 
given by
\begin{equation}\label{eq:C-action}
\begin{split}
Z_{\lambda(\sigma)}=&
\exp(-\pi  \sqrt{-1} \lambda)Z_\sigma\\
{\cal P}_{\lambda(\sigma)}(\phi)=&
{\cal P}_\sigma(\phi+\mathrm{Re}\lambda).
\end{split}
\end{equation}

\subsection{Geometric stability conditions}\label{subsect:geom}
Let $\Delta^+(X)$ be the set of  
$(-2)$-vectors $v_1 \in H^*(X,{\Bbb Z})_{\alg}$
with $\rk v_1 >0$.
For 
$$
(\beta, \omega) \in
{\cal H}:= 
\left(\NS(X)_{\Bbb R} \times \Amp(X)_{\Bbb R} \setminus 
\bigcup_{v_1 \in \Delta^+(X)}\left\{(\beta,\omega) \left|
\langle e^{\beta+\sqrt{-1} \omega},v_1 \rangle 
\in {\Bbb R}_{\leq 0} \right. \right\}\right),
$$
Bridgeland \cite{Br:3} constructed a stability condition
$\sigma_{(\beta,\omega)}=(Z_{(\beta,\omega)},{\cal A}_{(\beta,\omega)})$
which is characterized
by the stability of $k_x$ ($x \in X$), where
$Z_{(\beta,\omega)}(\bullet)=
\langle e^{\beta+\sqrt{-1}\omega},v(\bullet) \rangle:
{\bf D}(X) \to {\Bbb C}$ is the stability function and
${\cal A}_{(\beta,\omega)}$ 
is an abelian category which is a tilting of $\Coh(X)$
by a torsion pair $({\cal T}_{(\beta,\omega)},{\cal F}_{(\beta,\omega)})$:
\begin{enumerate}
\item
${\cal T}_{(\beta,\omega)}$ is a full subcategory of $\Coh(X)$ generated by
torsion sheaves and $\mu$-stable sheaves $E$ with 
$(c_1(E)-\rk E \beta,\omega)>0$ and
\item
${\cal F}_{(\beta,\omega)}$ is a full subcategory of $\Coh(X)$ generated by
$\mu$-stable sheaves $E$ with 
$(c_1(E)-\rk E \beta,\omega) \leq 0$.
\end{enumerate}
This kind of stability conditions are called geometric.
To be more precise,
for a geometric stability condition
$\sigma=(Z_\sigma,{\cal A}_\sigma)$,
we require that
$\mathrm{Re}\, \mho, \mathrm{Im}\,\mho$ span a positive definite
2-plane of $H^*(X,{\Bbb R})$,
where $Z_\sigma(\bullet)=\langle \mho,v(\bullet) \rangle$
with $\mho \in H^*(X,{\Bbb Q})_{\alg} \otimes {\Bbb C}$. 
Up to the action of $\widetilde{\GL}_2^+({\Bbb R})$, 
there is $(\beta,\omega)$ such that
$\sigma=\sigma_{(\beta,\omega)}$.
We set ${\cal M}_{(\beta,\omega)}(v):=
{\cal M}_{\sigma_{(\beta,\omega)}}(v)$.
For a primitive Mukai vector $v$ and a general $(\beta,\omega)$,
$M_{(\beta,\omega)}(v)$ denotes the moduli scheme of $\sigma_{(\beta,\omega)}$-stable
objects $E$ with $v(E)=v$.
\begin{NB}
Old:
$\sigma \in \Stab(X)$ satisfies  
(1)
$Z_\sigma(\bullet)=\langle \mho,\bullet \rangle$
with $\langle \mathrm{Re}\mho^2 \rangle=
\langle \mathrm{Im}\mho^2 \rangle$, 
$\langle \mathrm{Re}\mho,\mathrm{Im}\mho \rangle=0$
and
(2)
$k_x$ ($x \in X$) are $\sigma$-stable
with the phase $\phi_{\sigma}(k_x)=1$. Then
$\sigma=\sigma_{(\beta,\omega)}$ for some $(\beta,\omega)$.
\end{NB}
\begin{NB}
\begin{rem}
In order to construct a geometric stability condition $\sigma$
with $Z_\sigma=Z_{(\beta,\omega)}$,
we require that
$$
(\beta, \omega) \in 
\NS(X)_{\Bbb R} \times \Amp(X)_{\Bbb R} \setminus \bigcup_{v_1 \in \Delta(X)}
\left\{(\beta,\omega) \left| \langle e^{\beta+\sqrt{-1} \omega},v_1 \rangle 
\in {\Bbb R}_{\leq 0} \right. \right\},
$$
where $\Delta(X)$ is the set of  $(-2)$-vector of $H^*(X,{\Bbb Z})_{\alg}$.
\end{rem}
\end{NB}
For a fixed $\gamma \in \NS(X)_{\Bbb Q}$,
we consider a subspace of stability conditions
\begin{equation}
\{\sigma_{(\gamma+sH,tH)} \mid 
\text{ ${\Bbb C}_x$ is $\sigma_{(\gamma+sH,tH)}$-stable
for all $x \in X$}\}
\end{equation} 
which is parameterized by an open subset ${\Bbb H}^0$ of ${\Bbb H}$, where
\begin{equation}\label{eq:geometric}
{\Bbb H}^0={\Bbb H} \setminus \bigcup_{v_1 \in \Delta^+(X)} 
\left\{(s,t) \left| 
\langle e^{\gamma+(s+t \sqrt{-1})H},v_1 \rangle \in {\Bbb R}_{\leq 0}
\right. \right\}.
\end{equation}

\subsection{Line bundles associated to stability conditions.}\label{subsect:line-bundle}

We recall line bundles associated to stability conditions (\cite{BM}, \cite{MYY:2011:2}).
For a stability condition $\sigma$, we have a ${\Bbb R}$-divisor
$\theta_v(\mathrm{Im} \frac{\mho}{Z_\sigma (v)}) \in \NS(M_H(v))_{\Bbb R}$, 
where $Z_\sigma(\bullet)=\langle \mho,v(\bullet) \rangle$.
As in \cite{Y}, we introduce the following map.
\begin{defn}\label{defn:xi}
For a Mukai vector $v>0$ and $\gamma \in \NS(X)_{\Bbb Q}$, 
we define a map $\xi_{v,\gamma} :{\Bbb H} \to v^\perp$ by
\begin{equation}
\begin{split}
& \xi_{v,\gamma}(s,t):=\frac{1}{t}\mathrm{Im} \left(\langle e^{\gamma+(s-it)H},v \rangle e^{\gamma+(s+it)H} \right).
\end{split}
\end{equation}
\end{defn}

\begin{lem}
In the notation of subsection \ref{subsect:divisor},
we have the following description.
\begin{equation}
\frac{1}{t}\mathrm{Im} \left(\langle e^{\gamma+(s-it)H},v \rangle e^{\gamma+(s+it)H} \right)
=\left(r \frac{s^2+t^2}{2}(H^2)-a_\gamma \right)h_v+
((-sH+\delta) \cdot H)D_{v,\gamma}.
\end{equation}
In particular $\xi_{v,\gamma}({\Bbb H}) \subset L_{\Bbb R}$.
\end{lem}

\begin{proof}
By 
\begin{equation}
\langle e^{\gamma+(s-it)H},v \rangle =
s r(\delta \cdot H)-r\frac{s^2-t^2}{2}(H^2)-a_\gamma+itr((sH-\delta) \cdot H),
\end{equation}
we see that
\begin{equation*}
\begin{split}
& \frac{1}{t}\mathrm{Im} \left(\langle e^{\gamma+(s-it)H},v \rangle e^{\gamma+(s+it)H} \right)\\
=& \left(s r(\delta \cdot H)-r\frac{s^2-t^2}{2}(H^2)-a_\gamma \right)(H+s(H^2) \varrho_X)
+e^\gamma \left(1+sH+\frac{s^2-t^2}{2}(H^2) \varrho_X \right)r((sH-\delta) \cdot H)\\
=& \left(r \frac{s^2+t^2}{2}(H^2)-a_\gamma \right)e^\gamma (H+(\delta \cdot H)\rho_X)+
((sH-\delta) \cdot H)(r e^\gamma-a_\gamma \varrho_X)\\ 
=& \left(r \frac{s^2+t^2}{2}(H^2)-a_\gamma \right)h_v+
((-sH+\delta) \cdot H)D_{v,\gamma}.
\end{split}
\end{equation*}
\end{proof}

\begin{NB}
\begin{equation}
\begin{matrix}
\xi_v :& {\Bbb H} & \to & L_{\Bbb R} (\subset v^\perp)\\
& (s,t) & \mapsto &  
\frac{1}{t}\mathrm{Im} \left(\langle e^{\gamma+(s-it)H},v \rangle e^{\gamma+(s+it)H} \right)
\end{matrix}
\end{equation}
\end{NB}

\begin{NB}
$\overline{Z_{(\gamma+sH,tH)}(v)}e^{\gamma+(s+it)H}$
\end{NB}

\begin{lem}[{cf. \cite[Prop. 3.13]{Y}}]\label{lem:W_u}
For a wall $W_u (\subset {\Bbb H}^0)$ defined by a Mukai vector $u$,
$\xi_{v,\gamma}^{-1}(u^\perp)=W_u$.
\end{lem}

\begin{NB}
Assume that $(\delta \cdot H)>0$.
If $|s| \ll 1$, then it is 
movable.
\end{NB}

\begin{defn}\label{defn:Delta}
We take a Mukai vector $v>0$ and $\gamma \in \NS(X)_{\Bbb Q}$.
For a general $(s,t) \in {\Bbb H}^0$ with respect to $v$,
let $\Delta_{(s,t)}$ be a connected component of $v^\perp \setminus \cup_{u \in {\cal U}}u^\perp$
such that  $\xi_{v,\gamma}(s,t) \in \Delta_{(s,t)}$, where 
 ${\cal U}$ is the set of Mukai vectors $u$ in Theorem \ref{thm:BM2}
\end{defn}

\begin{ex}
Assume that $(\delta \cdot H)>0$.
If $t \gg 0$, then 
$\Delta_{(s,t)} \subset \theta_v^{-1}(\Mov(M_H(v)))$.
\end{ex}

\section{Fourier-Mukai transforms by a family of stable torsion free sheaves.}\label{sect:FM}

\subsection{Fourier-Mukai transforms and Bridgeland stability conditions. }\label{subsect:Bridgeland-FM}

Let $X$ be a K3 surface.
For a primitive and isotropic Mukai vector $v_0=r_0 e^\gamma$ $(r_0 >0)$
and a general ample divisor $H$ on $X$, we set
$X':=M_H(v_0)$.
Assume that $X'$ is a fine moduli space
which is a smooth K3 surface.
Let ${\cal E}$ be the universal family on
$X \times X'$.
Let us consider the equivalence
\begin{equation}
\Phi:=\Phi_{X\to X'}^{{\cal E}^{\vee}[1]}:{\bf D}(X) \to {\bf D}(X').
\end{equation}
We set
\begin{equation}\label{eq:H'}
H':=c_1(\overline{\Phi}_{X\to X'}^{{\cal E}^{\vee}[1]}(H e^\gamma))=
c_1(\overline{\Phi}_{X\to X'}^{{\cal E}[1]}(H e^{-\gamma})).
\end{equation}
%
Then we get the following commutative diagram of stability functions  
by \cite[sect. 5.1]{MYY:2011:2}:
\begin{equation}\label{eq:comm}
\xymatrix{
   {\bf D}(X) \ar[r]^{\Phi} \ar[d]_{Z_{(\gamma+sH,tH)}}
 & {\bf D}(X') \ar[d]^{Z_{(\gamma'+s'H',t'H')}} \\
   {\Bbb C}  \ar[r]_{\zeta^{-1}} 
 & {\Bbb C}
}
\end{equation}
where 
\begin{equation}\label{eq:zeta}
\begin{split}
s'=& \frac{2}{r_0 (H^2)}\frac{-s}{s^2+t^2}\\
t'=& \frac{2}{r_0 (H^2)}\frac{t}{s^2+t^2}\\
\zeta=& -r_0 
\frac{(s+\sqrt{-1}t)^2(H^2)}{2}\\
\gamma'=& -c_1({\cal E}_{|\{ x \} \times X'})/r_0.
\end{split}
\end{equation}
We set $(s',t'):=\Phi(s,t)$.

\begin{rem}
We set
$$
z:=s+t \sqrt{-1},\;z':=s'+t' \sqrt{-1}.
$$
Then
$$
zz'=-\frac{2}{r_0(H^2)}.
$$
\end{rem}

We also get the following duality
\begin{equation}
\begin{matrix}
X & \simto & M_{H'}(r_0 e^{\gamma'})\\
x & \mapsto & {\cal E}_{|\{ x \} \times X'}.
\end{matrix}
\end{equation}

\begin{NB}
\begin{rem}
The sign in 
\cite[A.1]{MYY:2011:2} is correct, however
the sign of $\widetilde{\beta}$
in \cite[(3.1)]{YY2}, \cite{movable} are not correct. 
\end{rem}
\end{NB}

\begin{NB}
\begin{prop}\label{prop:sigma}
Assume that $(\gamma+sH,tH)$ and $(\gamma+s_1 H,t_1 H)$ belong 
to the same chamber.
We also assume that ${\cal E}_{|X \times \{x' \}} 
\in {\cal A}_{(\gamma+s_1 H,t_1 H)}$
$(x' \in X')$
and $r_0>0$.
We set
$\sigma=\lambda(\Phi(\sigma_{(\gamma+sH,tH)})) \in \Stab(X')$
(see \eqref{eq:FM-action}, \eqref{eq:C-action}),
where 
$$
\zeta=\exp(\pi \sqrt{-1} \lambda),\; \mathrm{Re}(\lambda) \in (-2,0).
$$
Then
$k_{x'}$ ($x' \in X'$) is $\sigma$-stable with 
$\phi_\sigma(k_{x'})=1$.
Hence
$H'$ is ample and
$$
\Phi(\sigma_{(\gamma+sH,tH)})=
(-\lambda)( \sigma_{(\gamma'+s'H',t'H')}),
$$
\end{prop}
\end{NB}

\begin{defn}
\begin{enumerate}
\item[(1)]
Let ${\cal H}_-^\gamma \subset {\Bbb H}^0$ be the chamber for $r_0 e^\gamma$ such that
$$
(s,t) \in {\cal H}_-^\gamma \iff
M_{(\gamma+sH,tH)}(r_0 e^\gamma)=M_H(r_0 e^\gamma).
$$
Thus ${\cal H}_-^\gamma$ is the Gieseker chamber for $v_0$.
\item[(2)]
Let ${\cal H}_+^{\gamma} \subset {\Bbb H}^0$ be the chamber for $r_0 e^\gamma$ such that
$$
(s,t) \in {\cal H}_+^{\gamma} \iff
M_{(\gamma+sH,tH)}(r_0 e^\gamma)=
\{F \mid F^{\vee} \in M_{H'}(r_0 e^{-\gamma}) \}.
$$
Thus ${\cal H}_+^\gamma$ is the dual Gieseker chamber for $r_0 e^\gamma$.
\end{enumerate}
\end{defn}

For $(s,t) \in {\cal H}_-^\gamma$,
$$
\sigma_{(\gamma'+s'H',t'H')} \equiv \Phi_{X \to X'}^{{\cal E}^{\vee}[k]}(\sigma_{(\gamma+sH,tH)})
$$
modulo the action of $\widetilde{GL}_2^+({\Bbb R})$ on $\Stabd(X)$ (see \cite[Prop. 2.2]{FM-duality}).
We have a bijective map
\begin{equation}\label{eq:H_- 1}
\begin{matrix}
{\cal H}_-^\gamma & \to & {\cal H}_+^{\gamma'}\\
(s,t) & \mapsto & (s',t').
\end{matrix}
\end{equation}
Composing ${\cal D}_{X'}$ with \eqref{eq:H_- 1}, we also have a bijective map
\begin{equation}
\begin{matrix}
{\cal H}_-^\gamma & \to & {\cal H}_+^{\gamma'} & \to & {\cal H}_-^{-\gamma'}\\
(s,t) & \mapsto & (s',t') & \mapsto & (-s',t').
\end{matrix}
\end{equation}

We note that $M_H(r_0 e^{-\gamma}) \cong X'$.
Let ${\cal E}_1$ be a universal family on $X \times X'$.
Then 
$$
M_{(\gamma+sH,tH)}(r_0 e^\gamma)=\{({\cal E}_1)_{|X \times \{ x' \}}^{\vee} \mid x' \in X' \},\;
(s,t) \in {\cal H}_+^\gamma,
$$
where $r_0 e^{\gamma'}=v({\cal E}_{1|\{ x \}\times X'})$.
Hence 
$$
\sigma_{(\gamma'+s'H',t'H')} \equiv \Phi_{X \to X'}^{{\cal E}_1[k]}(\sigma_{(\gamma+sH,tH)}),\;(s,t) \in {\cal H}_+^\gamma
$$
modulo the action of $\widetilde{GL}_2^+({\Bbb R})$ on $\Stabd(X)$, and
we have a bijective map
\begin{equation}
\begin{matrix}
{\cal H}_+^\gamma & \to & {\cal H}_-^{\gamma'} & \to & {\cal H}_+^{-\gamma'}\\
(s,t) & \mapsto & (s',t') & \mapsto & (-s',t')
\end{matrix}
\end{equation}
by ${\cal D}_{X'} \circ \Phi_{X \to X'}^{{\cal E}_1[k]}$.

\begin{rem}\label{rem:U_pm}
Let ${\cal E}$ be the universal family of stable sheaves with the Mukai vector $v_0$.
By \cite[Lem. 1.2]{Y:Stability}, we can choose ${\cal E}$ to satisfy one of the following conditions:
\begin{enumerate}
\item
${\cal E}$ is a family of $\mu$-stable sheaves or 
\item
${\cal E}=\ker(E_0 \otimes E_0^{\vee} \to {\cal O}_\Delta)$, where
$E_0$ is a $\mu$-stable locally free sheaf with 
$\langle v(E_0)^2 \rangle=-2$.
\end{enumerate}
For the second case,
$\Psi:=\Phi_{X \to X}^{{\cal E}^{\vee}} \circ \Phi_{X \to X}^{{\cal E}^{\vee}}$
induces an isomorphism
$$
M_H(v_0) \simto M_{(\gamma+sH,tH)}(v_0),\;
(s,t) \in U_+^\gamma.
$$
In particular we may assume that ${\cal E}_1^{\vee}=\Psi({\cal E})$.
Then we have $\Phi_{X \to X}^{{\cal E}_1}=\Phi_{X \to X}^{{\cal E}^{\vee}} \circ \Psi^{-1}$
and $\Psi$ acts trivially on the Mukai lattice.
\end{rem}

We set
\begin{equation}
U_\pm^\gamma:=\left\{(s,t) \left|
0< \pm s<\min \left\{ \frac{1}{r_0}\sqrt{\frac{1}{2(H^2)}},
\frac{1}{2r_0^2} \right\}, 
t>\frac{1}{r_0}\sqrt{\frac{1}{2(H^2)}}-\sqrt{\frac{1}{2r_0^2 (H^2)}-s^2}
\right. \right\}.
\end{equation}
Then we have the following description.
\begin{prop}[{\cite[Prop. 2.11, Prop. 2.16]{FM-duality}}]\label{prop:stability-cond}
$$
U_-^\gamma \subset {\cal H}_-^\gamma \subset {\Bbb H}^0,\;
U_+^{\gamma} \subset {\cal H}_+^{\gamma} \subset {\Bbb H}^0.
$$
\end{prop}

\begin{NB}
\begin{equation}
\begin{matrix}
U_\pm^\gamma & \to & \Stabd(X)\\
(s,t) & \mapsto & \sigma_{(\gamma+sH,tH)}
\end{matrix}
\end{equation}
\end{NB}
%
\begin{NB}
We set $X'=M_H(r_0 e^\gamma)$, $r_0>0$.
Let ${\cal E}$ be a universal family of stable sheaves.

${\cal E}_{|X \times \{ x'\}}$ is $\sigma_{(\gamma+sH,tH)}$-stable object
if $(s,t) \in U_-^\gamma$ $(0<-s \ll 1)$.
$\Phi_{X \to X'}^{{\cal E}^{\vee}}(\sigma_{(\gamma+sH,tH)})$
\end{NB}
$\Phi_{X \to X'}^{{\cal E}^{\vee}[k]}$ induces an isomorphism
$$
{\cal M}_{(\gamma+sH,tH)}(v) \simto {\cal M}_{(\gamma'+s'H',t'H')}(v'),
$$
where $v'=\overline{\Phi}_{X \to X'}^{{\cal E}^{\vee}[k]}(v)>0$.
Let ${\cal C}_{-,0}^\gamma \subset U_-^\gamma$ be a chamber for $v$
whose closure contains $(0,0)$
and ${\cal C}_{+,0}^{\gamma'} \subset U_+^{\gamma'}$ be a chamber for $v'$
whose closure contains $(0,0)$.
We note that 
\begin{equation}
\begin{split}
t \gg 0 & \Longleftrightarrow (s',t' ) \in {\cal C}_{+,0}^{\gamma'},\\
t' \gg 0 & \Longleftrightarrow (s,t) \in {\cal C}_{-,0}^\gamma.
\end{split}
\end{equation}
Thus we have isomorphisms
\begin{equation}
\begin{split}
& {\cal M}_H(v) 
\simto  {\cal M}_{(\gamma'+s'H',t'H')}(v'),\quad   ((s',t') \in {\cal C}_{+,0}^{\gamma'})  \\
& {\cal M}_{(\gamma+sH,tH)}(v) 
\simto
{\cal M}_{H'}(v'),\quad   ((s,t) \in {\cal C}_{-,0}^\gamma).
\end{split}
\end{equation}

\begin{prop}\label{prop:G-chamber}
Let 
$$
v=r+\xi+a \varrho_X,
$$
be a Mukai vector with $r>0$.
\begin{enumerate}
\item
Assume that $((\xi-r \gamma-rsH)\cdot H)>0$. Then
$$
{\cal M}_{(\gamma+sH,tH)}(v)={\cal M}_H(v)
$$
if $t \gg 0$.
\item
Assume that $((\xi-r \gamma-rsH)\cdot H)<0$.
$$
{\cal M}_{(\gamma+sH,tH)}(v)=\{F \mid F^{\vee} \in {\cal M}_H(v^{\vee}) \}
$$
if $t \gg 0$.
\end{enumerate}
\end{prop}

\begin{rem}
Assume that $(\delta \cdot H)=0$.
Then
$$
{\cal M}_{(\gamma+sH,tH)}(v)={\cal M}_H(v)
$$
if $0<-s \ll 1$ and $t \gg 0$
and
$$
{\cal M}_{(\gamma+sH,tH)}(v)=\{F \mid F^{\vee} \in {\cal M}_H(v^{\vee}) \}
$$
if $0<s \ll 1$ and $t \gg 0$.
\end{rem}

\begin{cor}\label{cor:G-chamber}
Assume that $((\xi-r \gamma)\cdot H)>0$ or
$((\xi-r \gamma)\cdot H)<0$ and $\theta_v(h_v) \not \in \partial \Mov(M_H(v))$.
Then there is an autoequivalence $\Psi:{\bf D}(X) \to {\bf D}(X)$ 
such that $\overline{\Psi}$ acts trivially on $H^*(X,{\Bbb Z})$ and 
$\Psi$ induces an isomorphism 
${\cal M}_H(v) \to {\cal M}_{(\gamma+sH,tH)}(v)$
up to codimension 1, if $|s| \ll 1$ and $t \gg 0$.
\end{cor}

\begin{ex}
We set $v=(r,0,a)$ and assume that $0<-a<r$.
If $a=-1,-2$, then $\theta_v(h_v) \in \partial \Mov(M_H(v))$
Assume that $a<-2$.
For a general $E \in {\cal M}_H(r,0,a)$, we have an exact sequence
\begin{equation}
0 \to F \to E \to {\cal O}_X^{\oplus (r+a)} \to 0,
\end{equation}
where $F$ is a $\mu$-stable locally free sheaf.
For a general $E' \in {\cal M}_{(sH,tH)}(r,0,a)$, we have an exact sequence
\begin{equation}
0 \to  {\cal O}_X^{\oplus (r+a)} \to E' \to F \to 0,
\end{equation}
and $\Phi_{X \to X}^{I_\Delta [1]}(E')=F$ and $\Phi_{X \to X}^{I_\Delta [1]}(F)=E$.
Then $\Psi:=\Phi_{X \to X}^{I_\Delta [1]} \circ \Phi_{X \to X}^{I_\Delta [1]}$
satisfies the required properties.
\end{ex}

\subsection{$\xi_{v,\gamma}$ and isometries of Mukai lattices.}\label{subsect:f}

Let $X$ and $X'$ be K3 surfaces and
$f:H^*(X,{\Bbb Z}) \to H^*(X',{\Bbb Z})$ an orientation preserving Hodge isometry.
We take $v_0 \in H^*(X,{\Bbb Z})$ with $f(v_0)=\varrho_{X'}$.
We assume that $v_0=r_0 e^\gamma$ $(r_0>0)$.
Then we have
\begin{equation}
f(e^\gamma(x+\eta+y \varrho_X))=e^{\gamma'}(x/r_0 \varrho_{X'}-\widehat{\eta}+y r_0),
\end{equation}
where $\eta \in H^2(X,{\Bbb Q})$ and
$\widehat{\eta}:=-e^{-\gamma'} f(e^\gamma \eta) \in H^2(X',{\Bbb Q})$.
As in subsection \ref{subsect:divisor},
we set
\begin{equation}
v=r+\xi+a \varrho_X=e^\gamma (r+r\delta+a_\gamma \varrho_X).
\end{equation}

\begin{lem}\label{lem:Phi(h)}
Assume that $a_\gamma>0$.
\begin{enumerate}
\item[(1)]
We set
$$
v':=f(v)=r'+\xi'+a'\varrho_{X'}.
$$
Then $v'>0$ with 
$$
v'=e^{\gamma'}(r/r_0 \varrho_{X'}-r\widehat{\delta}+a_\gamma r_0).
$$
In particular 
\begin{equation}\label{eq:v'}
r'=a_\gamma r_0,\;
\xi'=a_\gamma r_0 \gamma'-r\widehat{\delta},\;
a'=a_\gamma r_0 \frac{({\gamma'}^2)}{2}-r (\widehat{\delta} \cdot \gamma')+\frac{r}{r_0}.
\end{equation}
\item[(2)]
$$
h_{v'}=e^{\gamma'}(\widehat{H}-\frac{r}{a_\gamma r_0}(\delta \cdot H)\varrho_{X'}),\;
D_{v',\gamma'}=-a_\gamma r_0 e^{\gamma'}+\frac{r}{r_0}\varrho_{X'}.
$$
\item[(3)]
\begin{equation}
f(h_v)=-h_{v'}-\frac{(\delta \cdot H)}{a_\gamma }D_{v',\gamma'},\;
f(D_{v,\gamma})= -D_{v',\gamma'}. 
\end{equation}
\end{enumerate}
\end{lem}

\begin{proof}
(1) is obvious.
(2) We note that 
$$
\frac{\xi'}{r'}-\gamma'=-\frac{r\widehat{\delta}}{r_0 a_\gamma}.
$$
Hence we get
$$
h_{v'}=e^{\gamma'}(\widehat{H}-\frac{r}{a_\gamma r_0}(\delta \cdot H)\varrho_{X'}),\;
D_{v',\gamma'}=-a_\gamma r_0 e^{\gamma'}+\frac{r}{r_0}\varrho_{X'}.
$$
(3) By (2), we see that
\begin{equation}\label{eq:Phi(h)}
\begin{split}
f(h_v)=& 
f(e^\gamma(H+(\delta \cdot H)\varrho_X) )\\
=& e^{\gamma'}(-\widehat{H}+(\delta \cdot H)r_0)\\
=& -h_{v'}-\frac{(\delta \cdot H)}{a_\gamma }D_{v',\gamma'},\\
f(D_{v,\gamma})=& -\frac{r}{r_0}\varrho_{X'}+a_\gamma r_0 e^{\gamma'}\\
=& -D_{v',\gamma'}.
\end{split}
\end{equation}
\end{proof}

In the same way as in the proof of Lemma \ref{lem:Phi(h)},
we get the following lemma.
\begin{lem}\label{lem:Phi(h)2}
Assume that $a_\gamma<0$ and set $v'=-f(v)$.
\begin{NB}
We set
$$
v'
=-f(v).
$$
Then $v'>0$ with 
$$
v'=e^{\gamma'}(-r/r_0 \varrho_{X'}+r\widehat{\delta}-a_\gamma r_0).
$$
In particular 
\begin{equation}\label{eq:v'}
r'=-a r_0,\;
\xi'=-a_\gamma r_0 \gamma'+r\widehat{\delta},\;
a'=-a_\gamma r_0 \frac{({\gamma'}^2)}{2}+r (\widehat{\delta} \cdot \gamma')-\frac{r}{r_0}.
\end{equation}
\end{NB}
\begin{enumerate}
\item[(1)]
$$
h_{v'}=e^{\gamma'}(\widehat{H}-\frac{r}{a_\gamma r_0}(\delta \cdot H)\varrho_{X'}),\;
D_{v',\gamma'}=a_\gamma r_0 e^{\gamma'}-\frac{r}{r_0}\varrho_{X'}.
$$
\item[(2)]
\begin{equation}
-f(h_v)=h_{v'}-\frac{(\delta \cdot H)}{a_\gamma }D_{v',\gamma'},\;
-f(D_{v,\gamma})= -D_{v',\gamma'}. 
\end{equation}
\end{enumerate}
\end{lem}

For the isometry $\overline{\cal D}_{X'} \circ f$, we also get the following two lemmas. 
\begin{lem}\label{lem:D*Phi(h)}
Assume that $a_\gamma>0$.
For the Mukai vector $v'=f(v)$ in Lemma \ref{lem:Phi(h)}, we set 
$$
v'':=\overline{\cal D}_{X'} \circ f(v)=e^{\gamma'}(r/r_0 \varrho_{X'}+r\widehat{\delta}+a_\gamma r_0).
$$
\begin{enumerate}
\item[(1)]
$$
h_{v''}=e^{-\gamma'}(\widehat{H}+\frac{r}{a_\gamma r_0}(\delta \cdot H)\varrho_{X'}),\;
D_{v'',-\gamma'}=-a_\gamma r_0 e^{-\gamma'}+\frac{r}{r_0}\varrho_{X'}.
$$
\item[(2)]
\begin{equation}
-\overline{\cal D}_{X'} \circ f(h_v)=-h_{v''}+\frac{(\delta \cdot H)}{a_\gamma }D_{v'',-\gamma'},\;
-\overline{\cal D}_{X'} \circ f(D_{v,\gamma})= D_{v'',-\gamma'}. 
\end{equation}
\end{enumerate}
\end{lem}

\begin{lem}\label{lem:D*Phi(h)2}
Assume that $a_\gamma<0$.
For the Mukai vector $v'=f(v)$ in Lemma \ref{lem:Phi(h)}, we set 
$$
v'':=-\overline{\cal D}_{X'} \circ f(v)=e^{\gamma'}(-r/r_0 \varrho_{X'}-r\widehat{\delta}-a_\gamma r_0).
$$
\begin{enumerate}
\item[(1)]
$$
h_{v''}=e^{-\gamma'}(\widehat{H}+\frac{r}{a_\gamma r_0}(\delta \cdot H)\varrho_{X'}),\;
D_{v'',-\gamma'}=a_\gamma r_0 e^{-\gamma'}-\frac{r}{r_0}\varrho_{X'}.
$$
\item[(2)]
\begin{equation}
-\overline{\cal D}_{X'} \circ f(h_v)= h_{v''}+\frac{(\delta \cdot H)}{a_\gamma }D_{v'',-\gamma'},\;
-\overline{\cal D}_{X'} \circ f(D_{v,\gamma})= D_{v'',-\gamma'}. 
\end{equation}
\end{enumerate}
\end{lem}

We define $(s',t')$ by \eqref{eq:zeta}. 
Then we have the following relation between $\xi_{v,\gamma}$ and
$\xi_{v',\gamma'}$.
\begin{lem}\label{lem:f(xi)}
Keep the notation in this subsection.
In particular $f$ is the orientation preserving Hodge isometry with $f^{-1}(\varrho_{X'})=r_0 e^\gamma$ 
$(r_0>0)$.
\begin{enumerate}
\item[(1)]
Assume that $v':=f(v)>0$.
\begin{equation}
f(\xi_{v,\gamma}(s,t))=\frac{2}{({s'}^2+{t'}^2)r_0 (H^2)}\xi_{v',\gamma'}(s',t').
\end{equation}
\item[(2)]
Assume that $v':=-f(v)>0$. 
\begin{equation}
-f(\xi_{v,\gamma}(s,t))=\frac{2}{({s'}^2+{t'}^2)r_0 (H^2)}\xi_{v',\gamma'}(s',t').
\end{equation}
\end{enumerate}
\end{lem}

\begin{proof}
For a convenience sake of the reader, we give a proof of (1).
We note that
\begin{equation}
f(e^{\gamma+zH})=\frac{z^2(H^2)r_0}{2}e^{\gamma'+z'H'}
\end{equation}
and
\begin{equation}
\frac{1}{t}=\frac{2}{(s^2+t^2)(H^2)r_0 t'},
\end{equation}
where $z=s+it$ and $z'=s'+it'$.
Hence
\begin{equation}
\begin{split}
f(\xi_{v,\gamma}(s,t))=&\frac{z^2(H^2)r_0}{2}\frac{\overline{z}^2(H^2)r_0}{2}
\frac{2}{|z|^2 (H^2)r_0}\xi_{v',\gamma'}(s',t')\\
=& \frac{|z|^2(H^2)r_0}{2}\xi_{v',\gamma'}(s',t')\\
=& \frac{2}{|z'|^2(H^2)r_0}\xi_{v',\gamma'}(s',t').
\end{split}
\end{equation}
\end{proof}

Since
\begin{equation}
\overline{\cal D}_{X}(\xi_{v,\gamma}(s,t))=-\xi_{v^\vee,-\gamma}(-s,t),
\end{equation}
we also get the following.
\begin{lem}\label{lem:Df(xi)}
\begin{enumerate}
\item[(1)]
Assume that $v'=\overline{\cal D}_{X'} \circ f(v)>0$.
Then
\begin{equation}
\overline{\cal D}_{X'} \circ f(\xi_{v,\gamma}(s,t))=-
\frac{2}{({s'}^2+{t'}^2)(H^2)r_0}\xi_{v',-\gamma'}(-s',t').
\end{equation}
\item[(2)]
Assume that $v'=-\overline{\cal D}_{X'} \circ f(v)>0$. Then
\begin{equation}
-\overline{\cal D}_{X'} \circ f(\xi_{v,\gamma}(s,t))=-
\frac{2}{({s'}^2+{t'}^2)(H^2)r_0}\xi_{v',-\gamma'}(-s',t').
\end{equation}
\end{enumerate}
\end{lem}

\begin{NB}
$\overline{D_{X'} \circ f(e^{\gamma+z H})}=\frac{\overline{z}^2 (H^2)r_0}{2} e^{-\gamma'-\overline{z}' H'}$.

\begin{equation}
\overline{\cal D}_X(\mathrm{Im} (\langle e^{\gamma+\overline{z}H},v \rangle e^{\gamma+zH})) 
=-\mathrm{Im} (\langle e^{-\gamma-zH},v \rangle e^{-\gamma-\overline{z}H}).
\end{equation}

Thus
\begin{equation}
\overline{\cal D}_{X}(\xi_{v,\gamma}(s,t))=-\xi_{v^\vee,-\gamma}(-s,t).
\end{equation}

For $v'=\overline{\cal D}_{X'} \circ f(v)$,
\begin{equation}
{\cal D}_{X'} \circ f(\xi_{v,\gamma}(s,t))=-
\frac{2}{({s'}^2+{t'}^2)(H^2)r_0}\xi_{v',-\gamma'}(-s',t').
\end{equation}

.

For $v'=-\overline{\cal D}_{X'} \circ f(v)$,
\begin{equation}
-\overline{\cal D}_{X'} \circ f(\xi_{v,\gamma}(s,t))=-
\frac{2}{({s'}^2+{t'}^2)(H^2)r_0}\xi_{v',-\gamma'}(-s',t').
\end{equation}

\end{NB}

\begin{NB}
We set $v'=D_{X'} \circ \Phi_{X \to X'}^{{\cal E}^{\vee}}(v)$.
Then
$$
v'=e^{-\gamma'}(r/r_0 \varrho_{X'}+r\widehat{\delta}+a_\gamma r_0).
$$
Hence
$$
D_{X'} \circ \Phi_{X \to X'}^{{\cal E}^{\vee}}(h_v)=h_{v'}-\frac{(\delta \cdot H)}{a_\gamma}D_{-\gamma'},
D_{X'} \circ \Phi_{X \to X'}^{{\cal E}^{\vee}}(D_\gamma)=-D_{-\gamma'}.
$$

$$
\frac{r}{r_0} \theta_{v'}(\xi_{v',-\gamma'}(0,0))=\theta_{v}(-D_{X'} \circ \Phi_{X \to X'}^{{\cal E}^{\vee}}(h_v)).
$$

\end{NB}

\begin{NB}
Assume that $a_\gamma>0$.

(1) We set $v':=\Phi_{X \to X'}^{{\cal E}^{\vee}}(v)$.
$\Phi_{X \to X'}^{{\cal E}^{\vee}}(-a_\gamma h_v+(\delta \cdot H)D_\gamma)=a_\gamma h_{v'}$.
If $\Phi_{X \to X'}^{{\cal E}^{\vee}}(\Mov(M_H(v)))=\Mov(M_{H'}(v'))$, then
$\theta_v(\xi_{v,\gamma}(0,0))=\theta_v(-a_\gamma h_v+(\delta \cdot H)D_\gamma)
\in \Mov(M_H(v))$.

(2)
We set $v'=D_{X'} \circ \Phi_{X \to X'}^{{\cal E}^{\vee}}(v)$.
$D_{X'} \circ \Phi_{X \to X'}^{{\cal E}^{\vee}}(-a_\gamma h_v+(\delta \cdot H)D_\gamma)=-a_\gamma h_{v'}$.
If $-D_{X'} \circ \Phi_{X \to X'}^{{\cal E}^{\vee}}(\Mov(M_H(v)))=\Mov(M_{H'}(v'))$, then
$\theta_v(\xi_{v,\gamma}(0,0))=\theta_v(-a_\gamma h_v+(\delta \cdot H)D_\gamma)
\in \Mov(M_H(v))$.

Assume that $a_\gamma<0$.

(3)
We set $v'=\Phi_{X \to X'}^{{\cal E}^{\vee}[1]}(v)$.
$\Phi_{X \to X'}^{{\cal E}^{\vee}[1]}(-a_\gamma h_v+(\delta \cdot H)D_\gamma)=-a_\gamma h_{v'}$.
If $\Phi_{X \to X'}^{{\cal E}^{\vee}[1]}(\Mov(M_H(v)))=\Mov(M_{H'}(v'))$, then
$\theta_v(\xi_{v,\gamma}(0,0))=\theta_v(-a_\gamma h_v+(\delta \cdot H)D_\gamma)
\in \Mov(M_H(v))$.

(4)
We set $v'=D_{X'} \circ \Phi_{X \to X'}^{{\cal E}^{\vee}[1]}(v)$.
$D_{X'} \circ \Phi_{X \to X'}^{{\cal E}^{\vee}[1]}(-a_\gamma h_v+(\delta \cdot H)D_\gamma)=a_\gamma h_{v'}$.
If $-D_{X'} \circ \Phi_{X \to X'}^{{\cal E}^{\vee}[1]}(\Mov(M_H(v)))=\Mov(M_{H'}(v'))$, then
$\theta_v(\xi_{v,\gamma}(0,0))=\theta_v(-a_\gamma h_v+(\delta \cdot H)D_\gamma)
\in \Mov(M_H(v))$.
\end{NB}

\subsection{Some constraints.}\label{subsect:r=1,2}

Let 
$v_0=r_0 e^\gamma$ $(r_0 >0)$ be a primitive isotropic Mukai vector and set
$X':=M_H(v_0)$.
Assume that $X'$ is a fine moduli space
which is a smooth K3 surface.
Let ${\cal E}$ be the universal family on
$X \times X'$.
We set $f:=\overline{\Phi}_{X \to X'}^{{\cal E}^{\vee}}$.
In this case $H'$ in \eqref{eq:H'} satisfies
$\widehat{H}=H'$.
Let 
$$
v=e^\gamma (r+r\delta+a_\gamma \varrho_X)
$$
be the Mukai vector in subsection \ref{subsect:divisor}.  
If $r$ is 1 or 2,
we have typical divisorial contractions of the moduli spaces
of stable sheaves.
Thus we have contractions of Hilbert-Chow or Donaldson-Uhlenbeck type.
In these cases, they define boundaries of the movable cones, which give
some constraints on $\langle v,v_0 \rangle$ and $(\delta \cdot H)$
for the preservation of the movable cones.

\begin{lem}\label{lem:a>0}
Assume that $\langle v_0,v \rangle<0$ and set $v':=f(v)$.
\begin{enumerate}
\item[(1)]
Assume that $\theta_v(h_v) \in \partial \Mov(M_H(v))$.
If $(\delta \cdot H)<0$, then
$$
f(\Mov(M_H(v)) \cap L) \not \subset \Mov(M_{H'}(v')).
$$
\item[(2)]
Assume that $\theta_{v'}(h_{v'}) \in \partial \Mov(M_{H'}(v'))$.
If $(\delta \cdot H)>0$, then
$$
f(\Mov(M_H(v)) \cap L) \not \subset \Mov(M_{H'}(v')).
$$
\end{enumerate}
\end{lem}

\begin{proof}
We keep the notation in subsection \ref{subsect:f}.
Assume that $a_\gamma>0$.
Then $v'>0$ and
$$
h_{v'}=e^{\gamma'}(\widehat{H}-\frac{r}{a_\gamma r_0}(\delta \cdot H)\varrho_{X'}),\;
D_{v',\gamma'}=-a_\gamma r_0 e^{\gamma'}+\frac{r}{r_0}\varrho_{X'}.
$$
By Lemma \ref{lem:Phi(h)},
we have the following representation matrix of 
$f:{\Bbb Q}h_v+{\Bbb Q}D_{v,\gamma} \to {\Bbb Q}h_{v'}+{\Bbb Q}D_{v',\gamma'}$:
$$
\begin{pmatrix}
f(h_v) & f(D_{v,\gamma})
\end{pmatrix}
=
\begin{pmatrix}
h_{v'} & D_{v',\gamma'}
\end{pmatrix}
\begin{pmatrix}
-1 & 0\\
-\frac{(\delta \cdot H)}{a_\gamma} & -1
\end{pmatrix}.
$$
We note the determinant of the matrix is 1.
\begin{NB}
\begin{equation}
\lim_{(s',t') \to (0,0)} \xi_{v',\gamma'}(s',t')=\frac{r}{r_0}\left( 
h_{D_{X'}(v')}-\frac{(\delta \cdot H)}{a_\gamma}D_{v',\gamma'} \right).
\end{equation}
\end{NB}

\begin{NB}
\begin{equation}
-{h_{v'}}^{\vee}=e^{-\gamma'}(\widehat{H}+\frac{r}{a_\gamma r_0}(\delta \cdot H)\varrho_{X'})
\end{equation}

\begin{equation}
D_{v',\gamma'}^{\vee}=D_{-\gamma'}.
\end{equation}

\begin{equation}
D_{X'} \circ \Phi^{{\cal E}^{\vee}}_{X \to X'}(e^{\xi/r}H)
=e^{-\gamma'}(\widehat{H}+\frac{r}{ar_0}(\delta \cdot H)\varrho_{X'})-
\frac{(\delta \cdot H)}{a_\gamma}D_{v',\gamma'}
\end{equation}
\end{NB}

(1)
Assume that $\theta_v(h_v) \in \partial \Mov(M_H(v))$ and
$(\delta \cdot H)<0$.
\begin{NB}
By Lemma \ref{lem:Phi(h)},
$$
\det \begin{pmatrix}
-1 & -\frac{(\delta \cdot H)}{a_\gamma}\\
0 & -1
\end{pmatrix}=1.
$$
\end{NB}
By Lemma \ref{lem:mov}, we get
$\theta_{v'}(h_{v'}) \not \in f(\Mov(M_H(v)))$.
Hence
$f(\Mov(M_H(v))) \not \subset \Mov(M_{H'}(v'))$.

\begin{NB}
$\langle \Phi^{{\cal E}^{\vee}}_{X \to X'}(xh_v+D_\gamma),\varrho_{X'} \rangle=
-\left(\frac{(\delta \cdot H)}{a}x +1 \right)a_\gamma r_0<0$, if $x<0$
\end{NB}

(2)
Assume that $\theta_{v'}(h_{v'}) \in \partial \Mov(M_{H'}(v'))$
and
$(\delta \cdot H)>0$.
Lemma \ref{lem:mov} implies
$f(h_v)$ is not movable.
Therefore the claim holds.
\end{proof}

\begin{lem}\label{lem:a<0}
Assume that $\langle v_0,v \rangle>0$ and set $v'=-f(v)$.
\begin{enumerate}
\item[(1)]
Assume that $\theta_v(h_v) \in \partial \Mov(M_H(v))$.
If $(\delta \cdot H) \leq 0$, then 
$$
-f(\Mov(M_H(v))) \not \subset \Mov(M_{H'}(v')).
$$ 
\item[(2)]
Assume that $\theta_{v'}(h_{v'}) \in \partial \Mov(M_{H'}(v'))$.
If $(\delta \cdot H) \leq 0$, then 
$$
-f(\Mov(M_H(v))) \not \subset \Mov(M_{H'}(v')).
$$ 
\end{enumerate}
\end{lem}

\begin{proof}
The proof is similar to the proof of Lemma \ref{lem:a>0}.
\begin{NB}
We set 
$$
v'':=-v'=r''+\xi''+a'' \varrho_{X'}.
$$

\begin{equation}
r''=-a_\gamma r_0,\; \xi''=-a_\gamma r_0\gamma'+r\widehat{\delta}.
\end{equation}

$$
\langle D_{v',\gamma'},\varrho_{X'} \rangle=a_\gamma r_0<0.
$$
Hence
$\Mov(M_H(v)) \cap L$ is spanned by
$h_v$ and $xh_v-D_{v,\gamma}$.

\begin{equation}
\begin{split}
-f(h_v)=& h_{v''}-\frac{(\delta \cdot H)}{a_\gamma}(-D_{v',\gamma'}),\\
-f(D_{v,\gamma})=& -(-D_{v',\gamma'}).
\end{split}
\end{equation}
\end{NB}
(1)
Assume that $\theta_v(h_v) \in \partial \Mov(M_H(v))$
and
$(\delta \cdot H) \leq 0$.
The representation matrix of $f$ satisfies
$$
\det \begin{pmatrix}
1 & 0\\
-\frac{(\delta \cdot H)}{a_\gamma} & -1
\end{pmatrix}=-1
$$
by Lemma \ref{lem:Phi(h)2}.
Hence we get
$\theta_{v'}(h_{v'}+\epsilon D_{v',\gamma'}) \not \in -f(\Mov(M_{H}(v)))$
$(0<\epsilon \ll 1)$.

(2)
Assume that $\theta_{v'}(h_{v'}) \in \partial \Mov(M_{H'}(v'))$.
Since $\langle -f(h_v+\epsilon D_{v,\gamma}),\varrho_{X'} \rangle<0$
for $0<\epsilon \ll 1$,
$-f(\Mov(M_H(v))) \not \subset \Mov(M_{H'}(v'))$. 
\end{proof}


\begin{NB}
$\theta_v^{-1}(\Mov(M_{H_1}(v))) \cap \theta_v^{-1}(\Mov(M_{H_2}(v)))=\emptyset$
implies $M_{H_1}(v)$ and $M_{H_2}(v)$ are separated by a codimension 0, 1 wall.
\end{NB}

\begin{NB}

\begin{prop}
Let $f:H^*(X,{\Bbb Z}) \to  H^*(X',{\Bbb Z})$ be an orientation preserving Hodge isometry. 
Assume that $f(r_0 e^\gamma)=\varrho_{X'}$ ($r_0>0$).
We set $\delta=\frac{c_1(v)}{r}-\gamma$.
Assume that $v':=\pm f(v)>0$ and $\pm f_{|v^\perp}$ is the isometry
\begin{NB2}
induced by a birational map $M_H(v) \cdots \to M_{H'}(v')$
\end{NB2}
inducing an isomorphism
$\Mov(M_H(v)) \to \Mov(M_{H'}(v'))$.
If $r=1$, $\langle v, e^\gamma \rangle \ne 0$ and
$(\delta \cdot H) \ne 0$, then $v'=-f(v)$ and
$(\delta \cdot H)>0$. 
\end{prop}

\begin{proof}
We set 
\begin{equation}
v=e^\gamma (r+r\delta+a_\gamma \varrho_X), \delta \in \NS(X)_{\Bbb Q}.
\end{equation}
Then $a_\gamma=-\langle v,e^\gamma \rangle$.
By Lemma \ref{lem:a>0} and Lemma \ref{lem:a<0},
$a_\gamma<0$ and $(\delta \cdot H)>0$.
In particular $-v'>0$.
\end{proof}

\end{NB}

\begin{lem}\label{lem:a>0:2}
Assume that $\langle v_0,v \rangle<0$ and set $v'':=\overline{\cal D}_{X'} \circ f(v)$.
\begin{enumerate}
\item[(1)]
Assume that $\theta_v(h_v)  \in \partial \Mov(M_H(v))$.
If $(\delta \cdot H)<0$, then
$$
-\overline{\cal D}_{X'} \circ f(\Mov(M_H(v)) \cap L) \not \subset \Mov(M_{H'}(v'')).
$$
\item[(2)]
Assume that $\theta_{v''}(h_{v''})  \in \partial \Mov(M_{H'}(v''))$.
If $(\delta \cdot H)<0$, then
$$
-\overline{\cal D}_{X'} \circ f(\Mov(M_H(v)) \cap L) \not \subset \Mov(M_{H'}(v'')).
$$
\end{enumerate}
\end{lem}

\begin{proof}
(1)
Assume that $\theta_v(h_v) \in \partial \Mov(M_H(v))$ and
$(\delta \cdot H)<0$.
By Lemma \ref{lem:D*Phi(h)}, the representation matrix of $f$ satisfies 
$$
\det \begin{pmatrix}
-1 & 0\\
\frac{(\delta \cdot H)}{a_\gamma} & 1
\end{pmatrix}=-1.
$$
By Lemma \ref{lem:mov}, we get
$\theta_{v''}(h_{v''}) \not \in -\overline{\cal D} \circ f(\Mov(M_H(v)))$.
Hence
$-\overline{\cal D} \circ f(\Mov(M_H(v))) \not \subset \Mov(M_{H'}(v''))$.

\begin{NB}
$\langle \Phi^{{\cal E}^{\vee}}_{X \to X'}(xh_v+D_\gamma),\varrho_{X'} \rangle=
-\left(\frac{(\delta \cdot H)}{a}x +1 \right)a_\gamma r_0<0$, if $x<0$
\end{NB}

(2)
Assume that $\theta_{v''}(h_{v''}) \in \partial \Mov(M_{H'}(v''))$
and $(\delta \cdot H)<0$.
By Lemma \ref{lem:D*Phi(h)}, Lemma \ref{lem:mov} implies
$-\overline{\cal D}_{X'} \circ f(h_v)$ is not movable.
Therefore the claim holds.
\begin{NB}
$\langle -\overline{\cal D}_{X'} \circ f (h_v),\varrho_{X'} \rangle=r_0 (\delta \cdot H)<0$.
\end{NB}
\end{proof}

\begin{lem}\label{lem:a<0:2}
Assume that $\langle v_0,v \rangle>0$ and set $v'':=-\overline{\cal D}_{X'} \circ f(v)$.
\begin{enumerate}
\item[(1)]
Assume that $\theta_v(h_v) \in \partial \Mov(M_H(v))$.
If $(\delta \cdot H)<0$, then
$$
\Mov(M_{H'}(v'') \not \subset \overline{\cal D}_{X'} \circ f(\Mov(M_H(v))).
$$
\item[(2)]
Assume that $\theta_{v''}(h_{v''}) \in \partial \Mov(M_{H'}(v''))$.
If $(\delta \cdot H)>0$, then
$$
\overline{\cal D}_{X'} \circ f(\Mov(M_H(v)) \cap L) \not \subset \Mov(M_{H'}(v'')).
$$
\end{enumerate}
\end{lem}

\begin{proof}
(1)
Assume that $\theta_v(h_v)  \in \partial \Mov(M_H(v))$ and
$(\delta \cdot H)<0$.
By Lemma \ref{lem:D*Phi(h)2}, the representation matrix of $f$ satisfies
$$
\det \begin{pmatrix}
1 & 0\\
\frac{(\delta \cdot H)}{a_\gamma} & 1
\end{pmatrix}=1.
$$
By Lemma \ref{lem:mov}, we get
$\theta_{v''}(h_{v''}) \not \in \overline{\cal D} \circ f(\Mov(M_H(v)))$.
Hence
$\overline{\cal D} \circ f(\Mov(M_H(v))) \not \subset \Mov(M_{H'}(v'))$.

\begin{NB}
$\langle \Phi^{{\cal E}^{\vee}}_{X \to X'}(xh_v+D_\gamma),\varrho_{X'} \rangle=
-\left(\frac{(\delta \cdot H)}{a}x +1 \right)a_\gamma r_0<0$, if $x<0$
\end{NB}

(2)
Assume that $\theta_{v''}(h_{v''})  \in \partial \Mov(M_{H'}(v''))$ and
$(\delta \cdot H)>0$.
By Lemma \ref{lem:D*Phi(h)}, Lemma \ref{lem:mov} implies
$\overline{\cal D}_{X'} \circ f(h_v)$ is not movable.
Therefore the claim holds.
\begin{NB}
$\langle -\overline{\cal D}_{X'} \circ f (h_v),\varrho_{X'} \rangle=r_0 (\delta \cdot H)<0$.
\end{NB}
\end{proof}

\begin{NB}

\begin{lem}\label{lem:a<0:3}
Assume that $\delta=0$, $\langle v_0,v \rangle>0$ and set $v'=-f(v)$.
\begin{enumerate}
\item[(1)]
Assume that $\theta_v(h_v)  \in \partial \Mov(M_H(v))$.
Then
$$
-f(\Mov(M_H(v)) \cap L) \not \subset \Mov(M_{H'}(v')).
$$
\item[(2)]
Assume that $\theta_{v''}(h_{v''})  \in \partial \Mov(M_{H'}(v''))$.
Then
$$
-f(\Mov(M_H(v)) \cap L) \not \subset \Mov(M_{H'}(v'')).
$$
\end{enumerate}
\end{lem}
\end{NB}

\section{Birational correspondences.}\label{sect:birat}

\subsection{Fourier-Mukai transforms and birational maps of moduli spaces.}\label{subsect:bir-FM}

Let $X$ be a K3 surface and
$v_0=(r_0,\eta,a_0)>0$ a primitive isotropic Mukai vector. 
Assume that $X':=M_H(v_0)$ is a K3 surface with a universal family
${\cal E}$.  
In this subsection, we shall show that $\Phi_{X \to X'}^{{\cal E}^{\vee}[k]}$ and 
$\Phi_{X \to X'}^{{\cal E}^{\vee}[k]} \circ {\cal D}_X$ induce birational maps
under some conditions on the behavior of $\theta_v^{-1}(\Mov(M_H(v))) \subset (v^\perp)_{\Bbb R}$
(Propositions \ref{prop:FM1} and \ref{prop:FM2}).  
We start with a result on an ample divisor of $M_H(0,\eta,a_0)$ ($\eta \ne 0$).

\begin{lem}\label{lem:r_0=0}
Assume that
$v_0=(0,\eta,a_0)$.
Then $\overline{\Phi}_{X \to X'}^{{\cal E}^{\vee}[1]}(-e^{\frac{a_0}{(\eta \cdot H)}H}+m(0,0,1))$ 
$(m \gg 0)$
defines an ample ${\Bbb Q}$-divisor on $X'$.
\end{lem}

\begin{proof}
Let $\Phi_{X \to X}^{I_\Delta(nH)}:{\bf D}(X) \to {\bf D}(X)$ be an equivalence
such that $\Phi_{X \to X}^{I_\Delta(nH)}(E)=\Phi_{X \to X}^{I_\Delta}(E(nH))$.
We take a sufficiently large integer $n$ such that
$\Phi_{X \to X}^{I_\Delta(nH)}$ induces
an isomorphism
$M_H(v_0) \to M_{H}(v_1)$,
where $v_1=(a_0+n(\eta \cdot H),-\eta,0)$.
We may assume that $M_H(v_1)$ consists of $\mu$-stable locally free sheaves.
Then the kernel ${\cal F}$ of the Fourier-Mukai transform 
$\Phi_{X \to X}^{I_\Delta(nH)} \circ \Phi_{X' \to X}^{{\cal E}}$
is a family of $\mu$-stable locally free sheaves with respect to $H$
(\cite{Y:Stability}). Thus
\begin{equation}
{\cal F}_{|X \times \{x' \}}=\Phi_{X \to X}^{I_\Delta}({\cal E}_{|X \times \{ x' \}}(nH)),\; x' \in X'.
\end{equation}
\begin{NB}
$v({\cal F}_{|X \times \{ x' \}})=(a+n(\eta \cdot H),-\eta,0)$.
\end{NB}
We note that
\begin{equation}
\begin{split}
&\overline{\Phi}_{X \to X}^{I_\Delta^{\vee}}((0,(a_0+n(\eta \cdot H))H,-(\eta \cdot H)))e^{-nH}\\
=& (-(\eta \cdot H),-a_0 H,(\eta \cdot H)n^2(H^2)/2+a_0 n(H^2))\\
=& -(\eta \cdot H) e^{\frac{a_0}{(\eta \cdot H)}H}+
\frac{(a_0+n(\eta \cdot H))^2}{2(\eta \cdot H)}(H^2)(0,0,1).
\end{split}
\end{equation}
Since $\theta_{v_1}((0,(a_0+n(\eta \cdot H))H,-(\eta \cdot H)))$ is ample,
we get our result.
\end{proof}
If $r_0>0$, then $v_0=r_0 e^\gamma$, where $\gamma=\eta/r_0$.
We take ${\Bbb Q}$-ample divisor $H'$ on $X'$ as follows: 
\begin{equation}\label{eq:polarization}
H':=
\begin{cases}
\theta_{v_0}(e^\gamma H),& v_0=r_0 e^\gamma, (r_0>0),\\
\theta_{v_0}(\zeta), & v_0=(0,\eta,a_0),\\
H,& v_0=(0,0,1),
\end{cases} 
\end{equation}
where 
$$
\zeta:=-\frac{(\eta \cdot H)}{a_0+n(\eta \cdot H)}e^{\frac{a_0}{(\eta \cdot H)}H}+\frac{(H^2)}{2}
\frac{a_0+n(\eta \cdot H)}{(\eta \cdot H)}(0,0,1),\; n \gg 0.
$$

\begin{rem}
Let $M_H(v_0) \to M_H(v_1)$ be the isomorphism
in the proof of Lemma \ref{lem:r_0=0}.
Under the identification
$H^2(M_H(v_1),{\Bbb Q}) \cong H^2(M_H(v_0),{\Bbb Q})$,
 $\theta_{v_1}(e^\gamma H)=\theta_{v_0}(\zeta)$, where 
$\gamma=-\frac{\eta}{a_0+n(\eta \cdot H)}$.
\end{rem}

\begin{NB}
$\theta_{v_0}(He^\gamma)=c_1(\Phi_{X \to X'}^{{\cal E}^{\vee}[1]}(He^\gamma))$.
\end{NB}

\begin{prop}\label{prop:FM1}
Let $\Phi_{X \to X'}^{{\cal E}^{\vee}}:{\bf D}(X) \to {\bf D}(X')$ be a Fourier-Mukai transform
such that ${\cal E}$ is a universal family of stable sheaves on $X$ with respect to $H$.
Assume that the cohomological transform preserves the movable cones:
\begin{equation}\label{eq:mov1}
\overline{\Phi}_{X \to X'}^{{\cal E}^{\vee}[k]}(\theta_v^{-1}(\Mov(M_H(v)))=\theta_{v'}^{-1}(\Mov(M_{H'}(v'))),
\end{equation}
where $H'$ is the ample divisor in \eqref{eq:polarization}.
Then there is an autoequivalence $\Psi:{\bf D}(X') \to {\bf D}(X')$
which acts trivially on $H^*(X,{\Bbb Z})$ and 
$\Phi:=\Psi \circ \Phi_{X \to X'}^{{\cal E}^{\vee}[k]}$ induces a birational map
$\varphi:M_H(v) \cdots \to M_{H'}(v')$ with a commutative diagram
\begin{equation}\label{eq:comm1}
\begin{CD}
v^\perp @>{\overline{\Phi}}>> {v'}^\perp\\
@V{\theta_v}VV @VV{\theta_{v'}}V\\
H^2(M_H(v),{\Bbb Z}) @>{\varphi_*}>> H^2(M_{H'}(v'),{\Bbb Z})
\end{CD}
\end{equation}
\end{prop}

\begin{proof}
Since $\overline{\Phi}_{X \to X'}^{{\cal E}^{\vee}[k]}$ preserves the orientation,
\eqref{eq:mov1} implies $\overline{\Phi}_{X \to X'}^{{\cal E}^{\vee}[k]}(v)=v'>0$.
We first assume that $r_0>0$.
We set 
\begin{equation}
\begin{split}
v= & e^\gamma (r+r \delta+a_\gamma \varrho_X)\\
v'=& (-1)^k e^{-\gamma'}(a_\gamma r_0-r \widehat{\delta}+r/r_0 \varrho_{X'}).
\end{split}
\end{equation}
If $(\delta \cdot H)<0$, then \eqref{eq:mov1}, Lemma \ref{lem:a>0} and Lemma \ref{lem:a<0} imply 
$\theta_v(h_v) \not \in \partial \Mov(M_H(v))$.
If $(-1)^k (\delta \cdot H)>0$, then \eqref{eq:mov1}, Lemma \ref{lem:a>0} and Lemma \ref{lem:a<0}
imply 
$\theta_{v'}(h_{v'}) \not \in \partial \Mov(M_{H'}(v'))$.
If $(\delta \cdot H)=0$, then $0<\langle v^2 \rangle =r^2 (\delta^2) -2r a_\gamma \leq -2r a_\gamma$.
Hence $a_\gamma<0$ and we get 
$\theta_v(h_v) \not \in \partial \Mov(M_H(v))$ and $\theta_{v'}(h_{v'}) \not \in \partial \Mov(M_{H'}(v'))$
by \eqref{eq:mov1} and Lemma \ref{lem:a<0}.

We take a path $I:[0,1] \to U_-^\gamma$ such that
(i) $I(\lambda):=(s_\lambda,t_\lambda)$ with $0<-s_\lambda \ll 1$,
(ii) $I(0)=(s_0,t_0)$ $t_0 \gg 0$ and (iii) $I(1)=(s_1,t_1) \in {\cal C}_{-,0}^\gamma$.
By Corollary \ref{cor:G-chamber}, we get the following.
\begin{enumerate}
\item
$M_{(\gamma+s_0 H,t_0 H)}(v)$ is isomorphic to $M_H(v)$ up to codimension 1 and
\item
$M_{(\gamma+s_1 H,t_1 H)}(v) \cong M_{(\gamma'+s_1' H',t_1' H')}(v')$ is isomorphic to $M_{H'}(v')$ up to codimension 1.
\end{enumerate}
In particular
\begin{equation}
\Delta_{(s_0,t_0)} \subset \theta_v^{-1}(\Mov(M_H(v))),\;
\Delta_{(s_1,t_1)} \subset \overline{\Phi}_{X' \to X}^{{\cal E}[2-k]}(\theta_{v'}^{-1}(\Mov(M_{H'}(v')))).
\end{equation}
\begin{NB}
\begin{equation}
\Mov(M_{(\gamma+s_0 H,t_0 H)}(v))=\Mov(M_H(v)),\;
\Mov(M_{(\gamma+s_1 H,t_1 H)}(v))=\overline{\Phi}_{X' \to X}^{{\cal E}[2-k]}(M_{H'}(v')).
\end{equation}
\end{NB}
Then \eqref{eq:mov1} implies
$\Delta_{(s_0,t_0)}=\Delta_{(s_1,t_1)}$.
\begin{NB}
By our assumption and Lemma \ref{lem:f(xi)},
\begin{equation}
-a_\gamma h_v+(\delta \cdot H)D_{v,\gamma}=\lim_{(s,t) \to (0,0)} \xi_{v,\gamma}(s,t) \in \theta_v^{-1}(\Mov(M_H(v))).
\end{equation}
\end{NB}
Hence we may assume that $I([0,1]) \subset \xi_{v,\gamma}^{-1}(\Delta_{(s_0,t_0)})$, which 
means that all walls correspond to small contractions or totally semistable walls associated to 
$(-2)$-vectors by Theorem \ref{thm:BM2} (see also Remark \ref{rem:BM2}).
We note that $\Phi_{X \to X'}^{{\cal E}^{\vee}[k]}$ induces a bijective map
$$
{\cal H}_-^\gamma \to {\cal H}_+^{\gamma'}
$$
which preserves chamber structures.
Then
there is an autoequivalence $\Psi:{\bf D}(X') \to {\bf D}(X')$
which acts trivially on $H^*(X',{\Bbb Z})$ and 
$\Phi:=\Psi \circ \Phi_{X \to X'}^{{\cal E}^{\vee}[k]}$ induces a birational map
$\varphi:M_H(v) \cdots \to M_{H'}(v')$
with a commutative diagram \eqref{eq:comm1}.
\begin{NB}
\begin{equation}
\begin{CD}
v^\perp @>{\Phi}>> {v'}^\perp\\
@V{\theta_v}VV @VV{\theta_{v'}}V\\
H^2(M_H(v),{\Bbb Z}) @>{\varphi_*}>> H^2(M_{H'}(v'),{\Bbb Z})
\end{CD}
\end{equation}
\end{NB}

We next assume that $v_0=(0,\eta,a_0)$, where $\eta$ is effective.
Let $\Phi_{X \to X}^{I_\Delta(nH)}:M_H(v_0) \to M_{H}(v_1)$
be the isomorphism and
${\cal F}$ the universal family of $\mu$-stable sheaves
in the proof of Lemma \ref{lem:r_0=0}.
Since $n \gg 0$, we may assume that
$\Phi_{X \to X}^{I_\Delta(nH)}:M_H(v) \to M_{H}(w)$
is the isomorphism \cite{Y:Stability}, where
$w=\overline{\Phi}_{X \to X}^{I_\Delta(nH)}(v)$.
By our assumption, we have an isomorphism
\begin{equation}
\theta_w^{-1}(\Mov(M_{H}(w))) \to
\theta_v^{-1}(\Mov(M_H(v))) \to 
\theta_{v'}^{-1}(\Mov(M_{H'}(v'))).
\end{equation}
Then there is an autoequivalence $\Psi:{\bf D}(X') \to {\bf D}(X')$
which acts trivially on $H^*(X,{\Bbb Z})$ and 
$\Psi \circ \Phi_{X \to X'}^{{\cal F}^{\vee}[k]}$ induces a birational map
$M_{H}(w) \cdots \to M_{H'}(v')$, which shows
that $\Psi \circ \Phi_{X \to X'}^{{\cal E}^{\vee}[k]}$ induces a desired birational map.

Finally we assume that $v_0=(0,0,1)$. In this case $X'=X$ and 
$$
\Phi_{X \to X}^{{\cal E}^{\vee}[2]}(F)=F \otimes L, F \in {\bf D}(X),
$$
where $L \in \Pic(X)$.
Since $H$ is a general polarization,
$\Phi_{X \to X}^{{\cal E}^{\vee}[2]}$ induces an isomorphism 
$M_H(v) \to M_H(v')$.
Hence we have an isomorphism
\begin{equation}
\theta_{v'}^{-1}(\Mov(M_H(v'))) \to 
\theta_{v'}^{-1}(\Mov(M_{H'}(v'))).
\end{equation}
Then there is an autoequivalence $\Psi:{\bf D}(X') \to {\bf D}(X')$
which acts trivially on $H^*(X,{\Bbb Z})$ and 
$\Psi \circ \Phi_{X \to X'}^{{\cal E}^{\vee}[2]}$ induces a birational map
$M_{H}(w) \cdots \to M_{H'}(v')$
\end{proof}

\begin{NB}
\begin{rem}
\eqref{eq:mov1} gives some constraints on $v$. 
Assume that $r_0>0$.
\begin{enumerate}
\item[(1)]
Assume that $r=1,2$.
By Lemma \ref{lem:a>0} and Lemma \ref{lem:a<0},
\eqref{eq:mov1} implies $(\delta \cdot H)>0$.
\item[(2)]
Assume that $\rk v'=1,2$.
Then $k$ is even and $(\delta \cdot H)<0$ or $k$ is odd and $(\delta \cdot H)>0$.
\end{enumerate}
\end{rem}
\end{NB}


\begin{prop}\label{prop:FM2}
Let $\Phi_{X \to X'}^{{\cal E}^{\vee}}:{\bf D}(X) \to {\bf D}(X')$ be a Fourier-Mukai transform
whose kernel is a universal family ${\cal E}$ of stable sheaves.
Assume that the cohomological transform preserves the movable cones:
\begin{equation}
-\overline{\cal D}_{X'} \circ \overline{\Phi}_{X \to X'}^{{\cal E}^{\vee}[k]}(\theta_v^{-1}(\Mov(M_H(v)))=\theta_{v'}^{-1}(\Mov(M_{H'}(v'))),
\end{equation}
where $H'$ is the ample divisor in \eqref{eq:polarization}.
Then there is an autoequivalence $\Psi:{\bf D}(X') \to {\bf D}(X')$
which acts trivially on $H^*(X,{\Bbb Z})$ and 
$\Phi:=\Psi \circ {\cal D}_{X'} \circ \Phi_{X \to X'}^{{\cal E}^{\vee}[k]}$ induces a birational map
$\varphi:M_H(v) \cdots \to M_{H'}(v')$ with a commutative diagram 
\begin{equation}\label{eq:comm2}
\begin{CD}
v^\perp @>{-\overline{\Phi}}>> {v'}^\perp\\
@V{\theta_v}VV @VV{\theta_{v'}}V\\
H^2(M_H(v),{\Bbb Z}) @>{\varphi_*}>> H^2(M_{H'}(v'),{\Bbb Z})
\end{CD}
\end{equation}
\end{prop}

\begin{proof}
We first assume that $r_0>0$.
We set 
\begin{equation}
\begin{split}
v= & e^\gamma (r+r \delta+a_\gamma \varrho_X)\\
v'=& (-1)^k e^{-\gamma'}(a_\gamma r_0+r \widehat{\delta}+r/r_0 \varrho_{X'}).
\end{split}
\end{equation}
If $(\delta \cdot H)<0$, then Lemma \ref{lem:a>0:2} implies $\theta_v(h_v) \not \in \partial \Mov(M_H(v))$.
If $(-1)^k (\delta \cdot H)<0$, then Lemma \ref{lem:a<0:2} implies $\theta_{v'}(h_{v'}) \not \in \partial \Mov(M_{H'}(v'))$.

${\cal D}_{X'} \circ \Phi_{X \to X'}^{{\cal E}^{\vee}[k]}$ induces an isomorphism
$$
{\cal H}_-^\gamma \to {\cal H}_-^{-\gamma'}.
$$
As in the proof of Proposition \ref{prop:FM1},
we take a path in $U_-^\gamma$ connecting $(s_0,t_0)$ $t_0 \gg 0$ and $(s_1,t_1) \in {\cal C}_{-,0}^\gamma$ such that
all walls correspond to small contractions or totally semistable walls associated to 
$(-2)$-vectors.
Then
there is an autoequivalence $\Psi:{\bf D}(X') \to {\bf D}(X')$
which acts trivially on $H^*(X,{\Bbb Z})$ and 
$\Phi:=\Psi \circ {\cal D}_{X'} \circ \Phi_{X \to X'}^{{\cal E}[k]}$ induces a birational map
$\varphi:M_H(v) \cdots \to M_{H'}(v')$
with a commutative diagram \eqref{eq:comm2}.

We next assume that $v_0=(0,\eta,a_0)$, where $\eta$ is effective.
As in the proof of Proposition \ref{prop:FM1},
let $\Phi_{X \to X}^{I_\Delta(nH)}:M_H(v_0) \to M_{H}(v_1)$
be the isomorphism and
${\cal F}$ the universal family of $\mu$-stable sheaves.
Since $n \gg 0$, we may assume that
$\Phi_{X \to X}^{I_\Delta(nH)}:M_H(v) \to M_{H}(w)$
is the isomorphism \cite{Y:Stability}, where
$w=\overline{\Phi}_{X \to X}^{I_\Delta(nH)}(v)$.
By our assumption, we have an isomorphism
\begin{equation}
\theta_w^{-1}(\Mov(M_{H}(w))) \to
\theta_v^{-1}(\Mov(M_H(v))) \to 
\theta_{v'}^{-1}(\Mov(M_{H'}(v'))).
\end{equation}
Then there is an autoequivalence $\Psi:{\bf D}(X') \to {\bf D}(X')$
which acts trivially on $H^*(X,{\Bbb Z})$ and 
$\Psi \circ \Phi_{X \to X'}^{{\cal F}^{\vee}[k]}$ induces a birational map
$M_{H}(w) \cdots \to M_{H'}(v')$, which shows
that $\Psi \circ \Phi_{X \to X'}^{{\cal E}^{\vee}[k]}$ induces a desired birational map.

\begin{NB}
We note that ${\cal E}^{\vee}[1]$ is a family of stable 1-dimensional sheaves with Mukai vector
$-v_0^{\vee}$.
Let $\Phi_{X \to X}^{I_\Delta(nH)}:M_H(-v_0^{\vee}) \to M_{H}(v_1^{\vee})$
be the isomorphism and
${\cal F}$ the universal family of $\mu$-stable sheaves. Thus
$$
{\cal F}_{|X \times \{ x' \}}=\Phi_{X \to X}^{I_\Delta}({\cal E}_{|X \times \{ x' \}}^{\vee}(nH)[1]).
$$
Since $n \gg 0$, we may assume that
$\Phi_{X \to X}^{I_\Delta(nH)}:M_H(v) \to M_{H}(w)$
is the isomorphism, where
$w=\overline{\Phi}_{X \to X}^{I_\Delta(nH)}(v)$.
By our assumption, we have an isomorphism
\begin{equation}
\theta_w^{-1}(\Mov(M_{H}(w)))
 \to \theta_v^{-1}(\Mov(M_H(v))) \to 
\theta_{v'}^{-1}(\Mov(M_{H'}(v'))).
\end{equation}
Hence $\Psi \circ {\cal D}_{X'} \circ \Phi_{X \to X'}^{{\cal F}^{\vee}[k+1]}$ induces a birational map
$M_{H}(w) \cdots \to M_{H'}(v')$, which shows
that $\Psi \circ {\cal D}_{X'} \circ \Phi_{X \to X'}^{{\cal E}[k]}$ induces a birational map.
\end{NB}
If $v_0=(0,0,1)$, then the proof is similar to the proof of Proposition \ref{prop:FM1}.
\end{proof}

\begin{rem}
For the Mukai vector $v$ in the proof of Proposition \ref{prop:FM2},
assume that $(\delta \cdot H) \ne 0$.
Then we have another formulation of Proposition \ref{prop:FM2}.
Let $\Phi_{X \to X'}^{{\cal E}_1}:{\bf D}(X) \to {\bf D}(X')$ be a Fourier-Mukai transform
whose kernel is a universal family ${\cal E}_1$ of stable sheaves with
$v({\cal E}_{|X \times \{ x' \}}^{\vee})=r_0 e^\gamma$.
Assume that the cohomological transform preserves the movable cones:
\begin{equation}
-\overline{\cal D}_{X'} \circ \overline{\Phi}_{X \to X'}^{{\cal E}_1 [k]}(\theta_v^{-1}(\Mov(M_H(v)))=\theta_{v'}^{-1}(\Mov(M_{H'}(v'))),
\end{equation}
where $H'$ is the ample divisor in \eqref{eq:polarization}.
Then there is an autoequivalence $\Psi_1:{\bf D}(X') \to {\bf D}(X')$
which acts trivially on $H^*(X,{\Bbb Z})$ and 
$\Phi:=\Psi_1 \circ {\cal D}_{X'} \circ \Phi_{X \to X'}^{{\cal E}_1[k]}$ induces a birational map
$\varphi:M_H(v) \cdots \to M_{H'}(v')$ with a commutative diagram \eqref{eq:comm2}.
By Remark \ref{rem:U_pm}, these two formulations are the same. 
 
For example, assume that $\Pic(X)={\Bbb Z}H$ and
$v=(r,H,a)$ with $r,a>0$. Then 
$D_X \circ \Phi_{X \to X}^{I_\Delta}(E)$ is stable for a general $E \in M_H(v)$
but $D_X \circ \Phi_{X \to X}^{I_\Delta^{\vee}}(E)$ is not stable for all $E \in M_H(v)$.
Thus we need to take a non-trivial autoequivalence $\Psi$ for
the equivalence $\Phi$, which shows that $D_X \circ \Phi_{X \to X}^{I_\Delta}$ is also useful
(cf. \cite{FM118}).
\end{rem}

\begin{NB}
We note that 
$$
D_{X'} \circ \Phi_{X \to X'}^{{\cal E}^{\vee}}(E)
={\bf R}\hom({\cal E},E)^{\vee}={\bf R}\Hom(E,{\cal E})[2].
$$
Hence the result is compatible with \cite[Prop. 2.8.1]{PerverseII}.

\end{NB}

\begin{NB}
\subsubsection{}

Let $f:H^*(X,{\Bbb Z}) \to  H^*(X',{\Bbb Z})$ be an orientation reversing Hodge isometry. 
Assume that $f(r_0 e^\gamma)=\varrho_{X'}$ ($r_0>0$).
We set $\delta=\frac{c_1(v)}{r}-\gamma$.
Assume that $v':=\pm f(v)>0$ and
$\mp f_{|v^\perp}$ is the isometry
\begin{equation}
\Mov(M_H(v)) \to \Mov(M_{H'}(v')).
\end{equation}
Then $v'=f(v)$ and
$(\delta \cdot H)>0$. 
\begin{NB2}
${\Bbb Q}v+{\Bbb Q}h_v \to H^*(X',{\Bbb Q}) \to {\Bbb Q}v'+{\Bbb Q}h_{v'}$
is orientation reversing. 
\end{NB2}
Let ${\cal D}_{X'} \circ \Phi_{X \to X'}^{{\cal E}^{\vee}}$ be a lifting of $f$.
We take a path connecting $(s,t)$ $t \gg 0$ and $(s,t) \in {\cal C}_0^\gamma$.
Since $f(\Mov(M_H(v))=\Mov(M_{H'}(v'))$, we may take a path such that
all walls correspond to small contractions or totally semistable walls associated to 
$(-2)$-vectors.
Then
there is an autoequivalence $\Psi:{\bf D}(X') \to {\bf D}(X')$
which acts trivially on $H^*(X,{\Bbb Z})$ and 
$\Phi:=\Psi \circ \Phi_{X \to X'}^{{\cal E}^{\vee}[1]}$ induces a birational map
$\varphi:M_H(v) \cdots \to M_{H'}(v')$
with a commutative diagram:
\begin{equation}
\begin{CD}
v^\perp @>{-\Phi}>> {v'}^\perp\\
@V{\theta_v}VV @VV{\theta_{v'}}V\\
H^2(M_H(v),{\Bbb Z}) @>{\varphi_*}>> H^2(M_{H'}(v'),{\Bbb Z})
\end{CD}
\end{equation}
\end{NB}

\begin{NB}

For $\Phi_{X \to X}^{I_\Delta (nH)}$,
we have a 

For $\sigma_{(sH,tH)}$,
$\sigma_{(\beta_1,\omega_1)} \sim \Phi_{X \to X}^{I_\Delta (nH)}(\sigma_{(sH,tH)})$, where
\begin{equation}
(\beta_1,\omega_1)=\left(-\frac{(s+n)}{(s+n)^2+t^2}\frac{2}{(H^2)}H,
\frac{(s+n)}{(s+n)^2+t^2}\frac{2}{(H^2)}H \right).
\end{equation}

Assume that $(\beta_1,\omega_1)$ belongs to the Gieseker chamber for $v_1$ and $u=r_0 e^{\beta_2}$,
where $u=v({\cal F}_{|X \times \{x' \}})$.  
Then we may replace $\beta_1$ be $\beta_2$.

 \end{NB}

\subsection{Birational maps induced by Fourier-Mukai transforms.}\label{subsect:main}

\begin{thm}\label{thm:bir-auto}
Let $(X,H)$ and $(X',H')$ be polarized K3 surfaces.
Let $\varphi:M_H(v) \cdots \to M_{H'}(v')$ be a birational
map, where $v \in H^*(X,{\Bbb Z})$ and $v' \in H^*(X',{\Bbb Z})$ are positive Mukai vectors.
Assume that $\Pic(X)={\Bbb Z}H$.
Then there is a (contravariant) Fourier-Mukai transform
$\Phi$ which induces $\varphi_*$.
\end{thm}

\begin{proof}
Let $\varphi:M_H(v) \to M_{H'}(v')$ be a birational map.
By the work of Markman \cite{Markman} on the monodromy group, 
we have a Hodge isometry
\begin{equation}
\varphi_*:H^2(M_H(v),{\Bbb Z}) \to H^2(M_{H'}(v'),{\Bbb Z})
\end{equation}
such that
\begin{enumerate}
\item
$\varphi_*(\Mov(M_H(v)))=\Mov(M_{H'}(v'))$ and
\item
$\varphi_*$ is extended to a Hodge isometry of Mukai lattices
$f:H^*(X,{\Bbb Z}) \to  H^*(X',{\Bbb Z})$ with
$f(v)=\pm v'$: Thus we have a commutative diagram.
\begin{equation}
\begin{CD}
v^\perp @>{f}>> {v'}^\perp\\
@V{\theta_v}VV @VV{\theta_{v'}}V\\
H^2(M_H(v),{\Bbb Z}) @>{\varphi_*}>> H^2(M_{H'}(v'),{\Bbb Z})
\end{CD}
\end{equation}
\end{enumerate}
$f$ is orientation preserving if and only if $f(v)=v'$.

(1) We first assume that $f$ preserves the orientation, that is,
$f(v)=v'$.
By Lemma \ref{lem:lifting} (1),
there is a Fourier-Mukai transform
$\Phi_{X \to X'}^{{\cal E}^{\vee}}:{\bf D}(X) \to {\bf D}(X')$
which induces $f$.
We take an isotropic Mukai vector $u$ with $f(u)=\varrho_{X'}$.
We set $r_0:=-\langle u,\varrho_X \rangle$.
If $r_0=0$, then
there is a line bundle ${\cal O}_X(dH)$ such that 
$\Phi_{X \to X'}^{{\cal E}^{\vee}}(E)=E(dH)[-2]$.
Hence $\Phi_{X \to X'}^{{\cal E}^{\vee}[2]}$ induces an isomorphism.

Assume that $r_0 \ne 0$ and set $u=r_0 e^\gamma$. Then
$X'=M_H((-1)^k u)$ and ${\cal E}[k]$ is the universal family of stable sheaves.
\begin{NB}
${\cal E}={\cal E}'[-k]$, ${\cal E}' \in \Coh(X \times X')$
and ${\cal E}^{\vee}=({\cal E}')^{\vee}[k]$, where $k=0,1,2$. 
\end{NB}
\begin{NB}
If $r_0>0$, then
$X'=M_H(u)$ and ${\cal E}$ is the universal family of stable sheaves.
If $r_0<0$, then 
$X'=M_H(-u)$ and ${\cal E}[-1]$ is the universal family of stable sheaves.
\end{NB}
Since $\rho(X')=\rho(X)=1$,
$\widehat{H} \in {\Bbb Q}_{>0}H'$.
By Proposition \ref{prop:FM1},
there is an autoequivalence $\Psi$ of ${\bf D}(X')$ acting trivially on the Mukai lattice
and
$\Psi \circ \Phi_{X \to X'}^{{\cal E}^{\vee}}$ induces the birational map
$\varphi_*:M_H(v) \cdots \to M_{H'}(v')$.

(2) We next assume that $f$ reverses the orientation, that is,
$f(v)=-v'$.
By Lemma \ref{lem:lifting}, there is a Fourier-Mukai transform
$\Phi_{X \to X'}^{{\cal E}^{\vee}}:{\bf D}(X) \to {\bf D}(X')$
such that ${\cal D}_{X'} \circ \Phi_{X \to X'}^{{\cal E}^{\vee}}$ induces $-f$.
We take an isotropic Mukai vector $u$ with $f(u)=\varrho_{X'}$ and
set $r_0:=-\langle u,\varrho_X \rangle$.
If $r_0=0$, then the claim holds as in (1).
Assume that $r_0 \ne 0$, and set
$u=r_0 e^\gamma$.
$X'=M_H((-1)^k u)$ and ${\cal E}[k]$ is the universal family of stable sheaves.
\begin{NB}
${\cal E}':={\cal E}[k]$ and $k=0,1,2$.
\end{NB}
\begin{NB}
If $r_0>0$, then
$X'=M_H(u)$ and ${\cal E}$ is the universal family of stable sheaves.
If $r_0<0$, then 
$X'=M_H(-u)$ and ${\cal E}[-1]$ is the universal family of stable sheaves.
\end{NB}
We note that
$\widehat{H} \in {\Bbb Q}_{>0}H'$.
By Proposition \ref{prop:FM2},
there is an autoequivalence $\Psi$ of ${\bf D}(X')$ acting trivially on the Mukai lattice
and
$\Psi \circ {\cal D}_{X'} \circ \Phi_{X \to X'}^{{\cal E}^{\vee}}$ induces the birational map
$\varphi_*:M_H(v) \cdots \to M_{H'}(v')$.
\begin{NB}
Then there is a (contravariant) Fourier-Mukai transform
$\Phi_{X \to X'}^{{\cal E}}:{\bf D}(X) \to {\bf D}(X')$
which induces a birational map $\varphi$.
$\varphi_*(\Mov(M_H(v))=\Mov(M_{H'}(v'))$.
We have a Hodge isometry
$f:H^*(X,{\Bbb Z}) \to H^*(X',{\Bbb Z})$.
\end{NB}
\end{proof}

\begin{NB}
\begin{rem}
If $r'=1$, then we may choose $H'=\widehat{H}$.
\begin{NB2}
If $r'>1$, then $H'$ and $\widehat{H}$ are different polarizations in general.
\end{NB2}
\end{rem}
\end{NB}

Since the stability does not depend on the choice of an ample divisor for the rank one case,
we also get the following result.
\begin{thm}\label{thm:bir-auto2}
Let $X$ be a K3 surface and assume that there is no $(-2)$-curves.
Then birational automorphisms of $\Hilb_X^{\ell+1}$ are induced by (contravariant) Fourier-Mukai transforms.  
\end{thm}

\begin{NB}
\begin{prob}
Assume that $\sigma$ and $\tau$ define divisors in the same chamber in $v^\perp$.
Do $\sigma$ and $\tau$ belong to the same chamber in $\Stabd(X)$?  
\end{prob}
\end{NB}

\section{Appendix}\label{sect:appendix}

In this appendix, we refer the uniqueness of extending an isometry
$H^2(M_H(v),{\Bbb Z}) \to H^2(M_{H'}(v'),{\Bbb Z})$ to an isometry of Mukai lattices. 
\begin{lem}\label{lem:unique-lift}
Let $f_i:H^*(X,{\Bbb Z}) \to H^*(X',{\Bbb Z})$ $(i=1,2)$ be Hodge isometries
such that $(f_1)_{|v^\perp}=(f_2)_{|v^\perp}$, where $v \in H^*(X,{\Bbb Z})_{\alg}$ is a 
primitive Mukai vector.
If $\langle v^2 \rangle \geq 4$, then $f_1=f_2$. 
\end{lem}

\begin{proof}
Since $f_1$ and $f_2$ are isometries, we have $f_1 (v)=\pm f_2(v)$.
We set $g:=f_2^{-1} \circ f_1$. Then $g_{|v^\perp}=1_{v^\perp}$
and $g(v)=\pm v$.
Since 
\begin{equation}\label{eq:A}
A_{v^\perp} \cong A_{{\Bbb Z}v}={\Bbb Z}v^{\vee}/{\Bbb Z}v \cong{\Bbb Z}/\langle v^2 \rangle {\Bbb Z}
\end{equation}
and the action of $g$ on $v^\perp$ is trivial,
we get $g(v)=v$. Thus $g$ is the identity map.
\end{proof}

\begin{lem}\label{lem:unique-lift2}
Let $v$ be a Mukai vector with $\langle v^2 \rangle=2$.
Then there is a Hodge isometry
$f:H^*(X,{\Bbb Z}) \to H^*(X,{\Bbb Z})$ 
such that $f_{|v^\perp}=1_{|v^\perp}$ and
$f(v)=-v$.
\end{lem}

\begin{proof}
Let $f$ be an isometry of $v^\perp \oplus {\Bbb Z}v$
such that $f_{|v^\perp}=1_{|v^\perp}$ and
$f(v)=-v$.
Since $\langle v^2 \rangle=2$, \eqref{eq:A} implies 
$A_{v^\perp} \cong {\Bbb Z}/2{\Bbb Z}$.
it is fiftable as a Hodge isometry of $H^*(X,{\Bbb Z})$.
\end{proof}

If $\NS(X)={\Bbb Z}H$, then
Theorem \ref{thm:bir-auto} implies 
there is a contravariant Fourier-Mukai transform
$\Phi$ which induces $f$ and $\Phi[1]$ induces a birational automorphism
of $M_H(v)$.
The following is an example of $\Phi$.
\begin{prop}[cf. {\cite[Prop. 2.8.1]{PerverseII}}]
For $I_Z \in \Hilb_X^2$,
$\Phi_{X \to X}^{I_\Delta[1]}(I_Z^{\vee})=I_Z$.
Thus $\Phi_{X \to X}^{I_\Delta[1]} \circ {\cal D}_X$ induces the identity map
$\Hilb_X^2 \to \Hilb_X^2$.
\end{prop}

\begin{NB}
Let $v$ be a primitive Mukai vector with $\langle v^2 \rangle=2$.
If $M_H(v)$ is birationally equivalent $\Hilb_{X'}^2$, then
there is a contravariant Fourier-Mukai transform $\Psi:{\bf D}(X) \to {\bf D}(X)$
which induces an identity map $M_H(v) \to M_H(v)$.
\end{NB}

\begin{NB}
For
$
\begin{pmatrix}
0 & 1\\
1 & 0
\end{pmatrix}
$,
we have an isomorphism
$$
M_H(1,0,-1) \cong M_H(1,0,-1)
$$
${\bf R}\Hom_p(I_Z,I_\Delta)[1] \cong I_Z$.

\begin{ex}
We set
$$
A=\begin{pmatrix}
\sqrt{2} & -1\\
1 & -\sqrt{2}
\end{pmatrix},\;
B=
\begin{pmatrix}
1 & -\sqrt{2}\\
\sqrt{2} & -1
\end{pmatrix}.
$$
Then 
$\phi_A(v)=v$ and $\phi_B(v)=-v$ 
$$
\phi_A(h)=\phi_B(h)=-3h+4\delta,\; \phi_A(\delta)=\phi_B(\delta)=-2h+3\delta
$$
where $h=(0,H,0)$ and
$\delta=(1,0,1)$.

$$
(\phi_A)_{|T}=(\phi_B)_{|T}=1_T.
$$

$\phi_A$ and $\phi_B$ define the same non-symplectic involution of $M_H(1,0,-1)$.
\end{ex}

We set $\ell:=\langle v^2 \rangle/2$.
 
If $\ell>2$, then FM is determined by the action on $v^\perp$. 

Indeed $A_{v^\perp} \cong A_{{\Bbb Z}v}={\Bbb Z}v^{\vee}/{\Bbb Z}v \cong{\Bbb Z}/2\ell {\Bbb Z}$.

$A_{H^*(X,{\Bbb Z})_{\alg}} \cong A_{T} \cong{\Bbb Z}/2n {\Bbb Z}$.

For
$
g=
\begin{pmatrix}
a \sqrt{n} & b\\
c & d \sqrt{n}
\end{pmatrix}
$ with $adn-bc=1$,
$g_{|T}=-1_T$.
Hence
if $g v=-v$, then 
it gives a symplectic action.

Note that if $H$ is general, then $\gamma$-twisted semistability is independent of the 
choice of $\gamma$.
Let ${\cal C}$ be a chamber for $v$.
For $v$ and $u=r_0 e^\gamma$,
we take a general ample divisor $H \in {\cal C}$.

Let $\Phi_{X \to X'}^{{\cal E}}$ be an equivalence.

Assume that $\Phi(v)=v'$,
Then $\Phi(P(v^\perp))=P({v'}^\perp)$.
$\Phi(v)=a r_0 e^\gamma+r/r_0 \varrho_{X'}+..$. 
Hence $ar_0>0$.

\begin{NB2}
Let $f$ be an orientation preserving Hodge isometry.
If $f(v)=v'$, then $f$ preserves the positive cones of $v^\perp$ and $f(v)^\perp$.
Then $f(rHe^\gamma+d(H^2)\varrho_X)=r_0 d(H^2)e^{\gamma'}-r_0/|r_0| rH' e^{\gamma'}$
$\langle f(rHe^\gamma+d(H^2)\varrho_X),$
\end{NB2}

Let $\Phi$ be an equivalence.
Then 
$\Phi(v)=ar_0 e^{\gamma'}+r/r_0 \varrho_{X'}-(dH'+D')e^{\gamma'}$.

If $a>0$, then 

We set $\Phi(u)=v({\Bbb C}_x)$.
Then $u=r_0 e^\gamma$.
$-r_0=\langle u,v({\Bbb C}_x) \rangle=\langle v({\Bbb C}_x),v(\Phi({\Bbb C}_x)) \rangle$ 

$u'=r_0 e^{\gamma'}$.

$v=r/r_0 u+a \varrho_X+e^\gamma (dH+D)$.

$\Phi(v)=r/r_0 \varrho_X+a u'-e^{\gamma'}(dH'+D')$.

If $\langle v, u \rangle<0$, then
$X'=M_H(r_0 e^{-\gamma})$ and
we consider $D_{X'} \circ \Phi_{X \to X'}^{{\cal E}}$.
\end{NB}

\begin{NB}

\begin{prop}\label{prop:FM3}
Assume that $(\delta \cdot H) \ne 0$.
Let $\Phi_{X \to X'}^{{\cal E}}:{\bf D}(X) \to {\bf D}(X')$ be a Fourier-Mukai transform
whose kernel is a universal family ${\cal E}$ of stable sheaves with
$v({\cal E}_{|X \times \{ x' \}}^{\vee})=r_0 e^\gamma$.
Assume that the cohomological transform preserves the movable cones:
\begin{equation}
-\overline{\cal D}_{X'} \circ \overline{\Phi}_{X \to X'}^{{\cal E}[k]}(\theta_v^{-1}(\Mov(M_H(v)))=\theta_{v'}^{-1}(\Mov(M_{H'}(v'))),
\end{equation}
where $H'$ is the ample divisor in \eqref{eq:polarization}.
Then there is an autoequivalence $\Psi:{\bf D}(X') \to {\bf D}(X')$
which acts trivially on $H^*(X,{\Bbb Z})$ and 
$\Phi:=\Psi \circ {\cal D}_{X'} \circ \Phi_{X \to X'}^{{\cal E}[k]}$ induces a birational map
$\varphi:M_H(v) \cdots \to M_{H'}(v')$ with a commutative diagram \eqref{eq:comm2}.
\end{prop}

\begin{proof}
We set 
\begin{equation}
\begin{split}
v'=& (-1)^k e^{-\gamma'}(a_\gamma r_0+r \widehat{\delta}+r/r_0 \varrho_{X'}).
\end{split}
\end{equation}
If $(\delta \cdot H)<0$, then Lemma \ref{lem:a>0:2} implies $\theta_v(h_v) \not \in \partial \Mov(M_H(v))$.
If $(-1)^k (\delta \cdot H)<0$, then Lemma \ref{lem:a<0:2} implies $\theta_{v'}(h_{v'}) \not \in \partial \Mov(M_{H'}(v'))$.

${\cal D}_{X'} \circ \Phi_{X \to X'}^{{\cal E}[k]}$ induces a bijection
\begin{equation}
{\cal H}_+^\gamma \to {\cal H}_+^{-\gamma'}.
\end{equation}
We take a path in $U_+^\gamma$ connecting $(s_0,t_0)$ $t_0 \gg 0$ and $(s_1,t_1) \in {\cal C}_{+,0}^\gamma$.
By our assumption,
we may take a path such that
all walls correspond to small contractions or totally semistable walls associated to 
$(-2)$-vectors.
Then
there is an autoequivalence $\Psi:{\bf D}(X') \to {\bf D}(X')$
which acts trivially on $H^*(X,{\Bbb Z})$ and 
$\Phi:=\Psi \circ {\cal D}_{X'} \circ \Phi_{X \to X'}^{{\cal E}[k]}$ induces a birational map
$\varphi:M_H(v) \cdots \to M_{H'}(v')$
with a commutative diagram \eqref{eq:comm2}.
\end{proof}

\end{NB}

\end{document}